\documentclass[11pt]{article}%
\usepackage{amsfonts}
\usepackage{amssymb}
\usepackage{amsmath}
\usepackage{graphicx}
\usepackage{theorem}
\usepackage[a4paper]{geometry}%
\providecommand{\U}[1]{\protect\rule{.1in}{.1in}}
\theoremstyle{break}
\newtheorem{theorem}{Theorem}[section]

\newtheorem{corollary}[theorem]{Corollary}

\newtheorem{lemma}[theorem]{Lemma}
\newtheorem{proposition}[theorem]{Proposition}
\theorembodyfont{\upshape}

\newtheorem{remark}[theorem]{Remark}

\newenvironment{proof}[1][Proof]{\noindent\textbf{#1.} }{\ \rule{0.5em}{0.5em}}
\numberwithin{equation}{section}
\makeindex
\begin{document}

\title{An iterative construction of solutions of the TAP equations for the
Sherrington-Kirkpatrick model}
\author{Erwin Bolthausen\thanks{Institute of Mathematics, University of Z\"{u}rich,
e-mail: eb@math.uzh.ch} \thanks{Supported by an SNF grant No 200020-125247,
and by the Humboldt Society.}
\and Universit\"{a}t Z\"{u}rich}
\date{}
\maketitle

\section{Introduction}

The TAP equations \cite{TAP} for the Sherrington-Kirkpatrick model describe
the quenched expectations of the spin variables in a large system.

The standard SK-model has the random Hamiltonian on $\Sigma_{N}\overset
{\mathrm{def}}{=}\left\{  -1,1\right\}  ^{N},\ N\in\mathbb{N},$%
\[
H_{N,\beta,h,\omega}\left(  \mathbf{\sigma}\right)  \overset{\mathrm{def}}%
{=}-\beta\sum_{1\leq i<j\leq N}g_{ij}^{\left(  N\right)  }\left(
\omega\right)  \sigma_{i}\sigma_{j}-h\sum_{i=1}^{N}\sigma_{i},
\]
where $\mathbf{\sigma}=\left(  \sigma_{1},\ldots,\sigma_{N}\right)  \in
\Sigma_{N},$ $\beta>0,\ h\geq0,$ and where the $g_{ij}^{\left(  N\right)  }$,
$1\leq i<j\leq N,$ are i.i.d. centered Gaussian random variables with variance
$1/N,$ defined on a probability space $\left(  \Omega,\mathcal{F}%
,\mathbb{P}\right)  .$ We extend this matrix to a symmetric one, by putting
$g_{ij}\overset{\mathrm{def}}{=}g_{ji}$ for $i>j,$ and $g_{ii}\overset
{\mathrm{def}}{=}0.$ The quenched Gibbs measure on $\Sigma_{N}$ is%
\[
\frac{1}{Z_{N,\beta,h,\omega}}\exp\left[  -H_{N,\beta,h,\omega}\left(
\mathbf{\sigma}\right)  \right]  ,
\]
where $Z_{N,\beta,h,\omega}\overset{\mathrm{def}}{=}\sum_{\mathbf{\sigma}}%
\exp\left[  -H_{N,\beta,h,\omega}\left(  \mathbf{\sigma}\right)  \right]  .$

We write $\left\langle \cdot\right\rangle _{N,\beta,h,\omega}$ for the
expectation under this measure. We will often drop the indices $N,\beta
,h,\omega$ if there is no danger of confusion. We set%
\[
m_{i}\overset{\mathrm{def}}{=}\left\langle \sigma_{i}\right\rangle .
\]

The TAP equations state that%
\begin{equation}
m_{i}=\tanh\left(  h+\beta\sum\nolimits_{j=1}^{N}g_{ij}m_{j}-\beta^{2}\left(
1-q\right)  m_{i}\right)  ,\label{TAP}%
\end{equation}
which have to be understood in a limiting sense, as $N\rightarrow\infty.$
$q=q\left(  \beta,h\right)  $ is the solution of the equation%
\begin{equation}
q=\int\mathrm{tanh}^{2}\left(  h+\beta\sqrt{q}z\right)  \phi\left(  dz\right)
,\label{fixed_point}%
\end{equation}
where $\phi\left(  dz\right)  $ is the standard normal distribution. It is
known that this equation has a unique solution $q>0$ for $h>0$ (see
\cite{TalaBuch} Proposition 1.3.8)$.$ If $h=0,$ then $q=0$ is the unique
solution if $\beta\leq1,$ and there are two other (symmetric) solutions when
$\beta>1,$ which are supposed to be the relevant ones. Mathematically, the
validity of the TAP equations has only been proved in the high temperature
case, i.e. when $\beta$ is small, although in the physics literature, it is
claimed that they are valid also at low temperature, but there they have many
solutions, and the Gibbs expectation has to be taken inside \textquotedblleft
pure states\textquotedblright. For the best mathematical results, see
\cite{TalaBuch} Chap. 1.7.

The appearance of the so-called Onsager term $\beta^{2}\left(  1-q\right)
m_{i}$ is easy to understand. From standard mean-field theory, one would
expect an equation%
\[
m_{i}=\tanh\left(  h+\beta\sum\nolimits_{j=1}^{N}g_{ij}m_{j}\right)  ,
\]
but one has to take into account the stochastic dependence between the random
variables $m_{j}$ and $g_{ij}.$ In fact, it turns out that the above equation
should be correct when one replaces $m_{j}$ by $m_{j}^{\left(  i\right)  }$
where the latter is computed under a Gibbs average dropping the interactions
with the spin $i.$ Therefore $m_{j}^{\left(  i\right)  }$ is independent of
the $g_{ik},\ 1\leq k\leq N,$ and one would get%
\begin{equation}
m_{i}=\tanh\left(  h+\beta\sum\nolimits_{j=1}^{N}g_{ij}m_{j}^{\left(
i\right)  }\right)  .\label{Stratonovich}%
\end{equation}
The Onsager term is an It\^{o}-type correction expanding the dependency of
$m_{j}$ on $g_{ji}=g_{ij},$ and replacing $m_{j}^{\left(  i\right)  }$ on the
right hand side by $m_{j}.$ The correction term is non-vanishing because of%
\[
\sum_{j}g_{ij}^{2}\approx1,
\]
i.e. exactly for the same reason as in the It\^{o}-correction in stochastic
calculus. We omit the details which are explained in \cite{MePaVi}.

In the present paper, there are no results about SK itself. We introduce an
iterative approximation scheme for solutions of the TAP equations which is
shown to converge below and at the de Almayda-Thouless line, i.e. under
condition (\ref{AT}) below (see \cite{AT}). This line is supposed to separate
the high-temperature region from the low-temperature one, but although the
full Parisi formula for the free energy of the SK-model has been proved by
Talagrand \cite{TalaAM}, there is no proof yet that the AT line is the correct
phase separation line.

The iterative scheme we propose reveals, we believe, an interesting structure
of the dependence of the $m_{i}$ on the family $\left\{  g_{ij}\right\}  ,$
even below the AT line. The main technical result, Proposition
\ref{Prop_Mhat_main} is proved at all temperatures, but beyond the AT-line, it
does not give much information.

We finish the section by introducing some notations.

If $\mathbf{x,\mathbf{y}}\in\mathbb{R}^{N},$ we write%
\[
\left\langle \mathbf{x},\mathbf{y}\right\rangle \overset{\mathrm{def}}{=}%
\frac{1}{N}\sum_{i=1}^{N}x_{i}y_{i},\ \left\Vert \mathbf{x}\right\Vert
\overset{\mathrm{def}}{=}\sqrt{\left\langle \mathbf{x},\mathbf{x}\right\rangle
}.
\]
As mentioned above, we suppress $N$ in notations as far as possible, but this
parameter is present everywhere.

We also define the $N\times N$-matrix $\mathbf{x}\otimes_{\mathrm{s}%
}\mathbf{y}$,\ $\mathbf{x}\otimes\mathbf{y,}$ by%
\begin{equation}
\left(  \mathbf{x}\otimes_{\mathrm{s}}\mathbf{y}\right)  _{i,j}\overset
{\mathrm{def}}{=}\frac{1}{N}\left(  x_{i}y_{j}+x_{j}y_{i}\right)
,\ \mathbf{x}\otimes\mathbf{y}\overset{\mathrm{def}}{=}\frac{x_{i}y_{j}}%
{N}.\label{Def_tens}%
\end{equation}

If $\mathbf{A}$ is an $N\times N$-matrix and $\mathbf{x}\in\mathbb{R}^{N},$
the vector $\mathbf{Ax}$ is defined in the usual way (interpreting vectors in
$\mathbb{R}^{N}$ as column matrices). If $f:\mathbb{R\rightarrow R}$ is a
function and $\mathbf{x}\in\mathbb{R}^{N}$ we simply write $f\left(
\mathbf{x}\right)  $ for the vector obtained by applying $f$ to the coordinates.

$\mathbf{g}=\left(  g_{ij}\right)  $ is a Gaussian $N\times N$-matrix where
the $g_{ij}$ for $i<j$ are independent centered Gaussians with variance $1/N$,
and where $g_{ij}=g_{ji},\ g_{ii}=0.$ We will exclusively reserve the notation
$\mathbf{g}$ for such a Gaussian matrix.

We will use $Z,Z^{\prime},Z_{1},Z_{2},\ldots$ as generic standard Gaussians.
Whenever several of them appear in the same formula, they are assumed to be
independent, without special mentioning. We then write $E$ when taking
expectations with respect to them. (This notation is simply an outflow of the
abhorrence probabilists have of using integral signs, as John Westwater once
put it).

If $\left\{  X_{N}\right\}  ,\left\{  Y_{N}\right\}  $ are two sequences of
real random variables, defined on $\left(  \Omega,\mathcal{F},\mathbb{P}%
\right)  $, we write%
\[
X_{N}\simeq Y_{N}%
\]
provided there exists a constant $C>0$ such that%
\[
\mathbb{P}\left(  \left\vert X_{N}-Y_{N}\right\vert \geq t\right)  \leq
C\exp\left[  -t^{2}N/C\right]
\]
for all $N\in\mathbb{N},\ 0<t\leq1.$

Clearly, if $X_{N}\simeq Y_{N},$ and $X_{N}^{\prime}\simeq Y_{N}^{\prime}$,
then $X_{N}+X_{N}^{\prime}\simeq Y_{N}+Y_{N}^{\prime}.$

If $\mathbf{X}^{\left(  N\right)  }=\left(  X_{i}^{\left(  N\right)  }\right)
_{i\leq N},\ \mathbf{Y}^{\left(  N\right)  }=\left(  Y_{i}^{\left(  N\right)
}\right)  _{i\leq N}$ are two sequences of random vectors in $\mathbb{R}^{N},$
we write $\mathbf{X}^{\left(  N\right)  }\approx\mathbf{Y}^{\left(  N\right)
}$ if%
\[
\frac{1}{N}\sum_{i=1}^{N}\left\vert X_{i}^{\left(  N\right)  }-Y_{i}^{\left(
N\right)  }\right\vert \simeq0.
\]

We will use $C>0$ as a generic positive constant, not necessarily the same at
different occurrences. It may depend on $\beta,h,$ and on the level $k$ of the
approximation scheme appearing in the next section, but on nothing else,
unless stated otherwise.

In order to avoid endless repetitions of the parameters $h$ and $\beta,$ we
use the abbreviation%
\[
\operatorname*{Th}\left(  x\right)  =\tanh\left(  h+\beta x\right)  .
\]

We always assume $h\neq0,$ and as there is a symmetry between the signs, we
assume $h>0.$ $q=q\left(  \beta,h\right)  $ will exclusively be used for the
unique solution of (\ref{fixed_point}). In the case $h=0,\ \beta>1,$ there is
a unique solution of (\ref{fixed_point}) which is positive. Proposition
\ref{Prop_Mhat_main} is valid in this case, too, but this does not lead to a
useful result. So, we stick to the $h>0$ case.

Gaussian random variables are always assumed to be centered.

\section{The recursive scheme for the solutions of the TAP equations}

We recursively define a double sequences $\mathbf{m}^{\left(  k\right)
}=\left\{  m_{i}^{\left(  k\right)  }\right\}  _{1\leq i\leq N,\ k\in
\mathbb{N}} $ of random variables by putting%
\[
\mathbf{m}^{\left(  0\right)  }\overset{\mathrm{def}}{=}0,\ \mathbf{m}%
^{\left(  1\right)  }\overset{\mathrm{def}}{=}\sqrt{q}\mathbf{1},\ \forall
i\leq N,
\]
$\mathbf{1}$ here the vector with coordinates all $1,$ and $q=q\left(
\beta,h\right)  $ is the unique solution of (\ref{fixed_point}). We define%
\[
\mathbf{m}^{\left(  k+1\right)  }\overset{\mathrm{def}}{=}\operatorname*{Th}%
\left(  \mathbf{g\mathbf{m}}^{\left(  k\right)  }-\beta\left(  1-q\right)
\mathbf{m}^{\left(  k-1\right)  }\right)  ,~k\geq1.
\]
$k$ will exclusively been used to number this level of the iteration. Our main
result is

\begin{theorem}
\label{Th_Main1}Assume $h>0.$ If $\beta>0$ is below the AT-line, i.e. if%
\begin{equation}
\beta^{2}E\cosh^{-4}\left(  h+\beta\sqrt{q}Z\right)  \leq1,\label{AT}%
\end{equation}
then%
\[
\lim_{k,k^{\prime}\rightarrow\infty}\limsup_{N\rightarrow\infty}%
\mathbb{E}\left\Vert \mathbf{m}^{\left(  k\right)  }-\mathbf{m}^{\left(
k^{\prime}\right)  }\right\Vert ^{2}=0.
\]
If there is strict inequality in (\ref{AT}), then there exist $0<\lambda
\left(  \beta,h\right)  <1,$ and $C>0,$ such that for all $k$%
\[
\limsup_{N\rightarrow\infty}\mathbb{E}\left\Vert \mathbf{m}^{\left(
k+1\right)  }-\mathbf{m}^{\left(  k\right)  }\right\Vert ^{2}\leq C\lambda
^{k}.
\]

\end{theorem}

The theorem is a straightforward consequence of a computation of the inner
products $\left\langle \mathbf{m}^{\left(  i\right)  },\mathbf{m}^{\left(
j\right)  }\right\rangle .$ We explain that first. The actual computation of
these inner products will be quite involved and will depend on clarifying the
structural dependence of $\mathbf{m}^{\left(  k\right)  }$ on $\mathbf{g.}$

As we assume $h>0,$ we have $q>0.$ We define a function $\psi:\left[
0,q\right]  \rightarrow\mathbb{R}$ by%
\[
\psi\left(  t\right)  \overset{\mathrm{def}}{=}E\operatorname*{Th}\left(
\sqrt{t}Z+\sqrt{q-t}Z^{\prime}\right)  \operatorname*{Th}\left(  \sqrt
{t}Z+\sqrt{q-t}Z^{\prime\prime}\right)  ,
\]
where $Z,Z^{\prime},Z^{\prime\prime},$ as usual, are independent standard
Gaussians. Remember that $\operatorname*{Th}\left(  x\right)  =\tanh\left(
h+\beta x\right)  .$

Let $\alpha\overset{\mathrm{def}}{=}E\operatorname*{Th}\left(  \sqrt
{q}Z\right)  >0. $

\begin{lemma}
\label{Le_psi}

\begin{enumerate}
\item[a)] $\psi$ satisfies $0<\psi\left(  0\right)  =\alpha^{2}<\psi\left(
q\right)  =q,$ and is strictly increasing and convex on $\left[  0,q\right]
.$

\item[b)]
\[
\psi^{\prime}\left(  q\right)  =\beta^{2}E\cosh^{-4}\left(  h+\beta\sqrt
{q}Z\right)  .
\]

\end{enumerate}
\end{lemma}

\begin{proof}
$\psi\left(  0\right)  =\alpha^{2},$ and $\psi\left(  q\right)  =q$ are
evident by the definition of $\alpha,q.$ We compute the first two derivatives
of $\psi:$%
\begin{align*}
\psi^{\prime}\left(  t\right)   & =\frac{1}{\sqrt{t}}%
E\Big [Z\operatorname*{Th}\nolimits^{\prime}\left(  \sqrt{t}Z+\sqrt
{q-t}Z^{\prime}\right)  \operatorname*{Th}\left(  \sqrt{t}Z+\sqrt
{q-t}Z^{\prime\prime}\right)  \Big ]\\
& -\frac{1}{\sqrt{q-t}}E\Big [Z^{\prime}\operatorname*{Th}\nolimits^{\prime
}\left(  \sqrt{t}Z+\sqrt{q-t}Z^{\prime}\right)  \operatorname*{Th}\left(
\sqrt{t}Z+\sqrt{q-t}Z^{\prime\prime}\right)  \Big ]\\
& =E\operatorname*{Th}\nolimits^{\prime\prime}\left(  \sqrt{t}Z+\sqrt
{q-t}Z^{\prime}\right)  \operatorname*{Th}\left(  \sqrt{t}Z+\sqrt
{q-t}Z^{\prime\prime}\right) \\
& +E\operatorname*{Th}\nolimits^{\prime}\left(  \sqrt{t}Z+\sqrt{q-t}Z^{\prime
}\right)  \operatorname*{Th}\nolimits^{\prime}\left(  \sqrt{t}Z+\sqrt
{q-t}Z^{\prime\prime}\right) \\
& -E\operatorname*{Th}\nolimits^{\prime\prime}\left(  \sqrt{t}Z+\sqrt
{q-t}Z^{\prime}\right)  \operatorname*{Th}\left(  \sqrt{t}Z+\sqrt
{q-t}Z^{\prime\prime}\right) \\
& =E\operatorname*{Th}\nolimits^{\prime}\left(  \sqrt{t}Z+\sqrt{q-t}Z^{\prime
}\right)  \operatorname*{Th}\nolimits^{\prime}\left(  \sqrt{t}Z+\sqrt
{q-t}Z^{\prime\prime}\right)  .
\end{align*}
the second equality by Gaussian partial integration.

Differentiating once more, we get%
\[
\psi^{\prime\prime}\left(  t\right)  =E\left(  \operatorname*{Th}%
\nolimits^{\prime\prime}\left(  \sqrt{t}Z+\sqrt{q-t}Z^{\prime}\right)
\operatorname*{Th}\nolimits^{\prime\prime}\left(  \sqrt{t}Z+\sqrt
{q-t}Z^{\prime\prime}\right)  \right)  .
\]

In both expressions, we can first integrate out $Z^{\prime},Z^{\prime\prime},$
getting%
\[
\psi^{\prime}\left(  t\right)  =\int_{-\infty}^{\infty}\left[  \int_{-\infty
}^{\infty}\operatorname*{Th}\nolimits^{\prime}\left(  \sqrt{t}x+\sqrt
{q-t}y\right)  \phi\left(  dy\right)  \right]  ^{2}\phi\left(  dx\right)  >0,
\]
and the similar expression for $\psi^{\prime\prime}$ with $\operatorname*{Th}%
\nolimits^{\prime}$ replaced by $\operatorname*{Th}\nolimits^{\prime\prime}.$
So, we see that $\psi$ is increasing and convex. Furthermore, as%
\begin{align*}
\operatorname*{Th}\nolimits^{\prime}\left(  x\right)   & =\beta\tanh^{\prime
}\left(  \beta x+h\right)  =\beta\left(  1-\tanh^{2}\left(  \beta x+h\right)
\right) \\
& =\frac{\beta}{\cosh^{2}\left(  \beta x+h\right)  },
\end{align*}
we get%
\[
\psi^{\prime}\left(  q\right)  =E\operatorname*{Th}\nolimits^{\prime}\left(
\sqrt{q}Z\right)  ^{2}=\beta^{2}E\cosh^{-4}\left(  h+\beta\sqrt{q}Z\right)  .
\]

\end{proof}

\begin{corollary}
If (\ref{AT}) is satisfied, then $q$ is the only fixed point of $\psi$ in the
interval $\left[  0,q\right]  .$ If (\ref{AT}) is not satisfied then there is
a unique fixed point of $\psi\left(  t\right)  =t$ inside the interval
$\left(  0,q\right)  .$
\end{corollary}

We define sequences $\left\{  \rho_{k}\right\}  _{k\geq1},\ \left\{
\gamma_{k}\right\}  _{k\geq1}$ recursively by $\gamma_{1}\overset
{\mathrm{def}}{=}\alpha,\ \rho_{1}\overset{\mathrm{def}}{=}\gamma_{1}\sqrt
{q},\ $and for $k\geq2$%
\[
\rho_{k}\overset{\mathrm{def}}{=}\psi\left(  \rho_{k-1}\right)  ,
\]%
\[
\gamma_{k}\overset{\mathrm{def}}{=}\frac{\rho_{k}-\Gamma_{k-1}^{2}}%
{\sqrt{q-\Gamma_{k-1}^{2}}},
\]
where%
\[
\Gamma_{m}^{2}\overset{\mathrm{def}}{=}\sum_{j=1}^{m}\gamma_{j}^{2}%
,\ \Gamma_{0}^{2}\overset{\mathrm{def}}{=}0.
\]

\begin{lemma}
\label{Le_ConvergenceConstants}

\begin{enumerate}
\item[a)] For all $k\in\mathbb{N}$%
\[
\Gamma_{k-1}^{2}<\rho_{k}<q.
\]

\item[b)] If (\ref{AT}) is satisfied, then%
\[
\lim_{k\rightarrow\infty}\rho_{k}=q,\ \lim_{k\rightarrow\infty}\Gamma_{k}%
^{2}=q.
\]

\item[c)] If there is strict inequality in (\ref{AT}) , then $\Gamma_{k}^{2} $
and $\rho_{k}$ converge to $q$ exponentially fast.
\end{enumerate}
\end{lemma}

\begin{proof}
a) $\rho_{k}<q$ for all $k$ is evident.

We prove by induction on $k$ that $\rho_{k}>\Gamma_{k-1}^{2}.$ For $k=1,$ as
$\rho_{1}=\gamma_{1}\sqrt{q},$ the statement follows.

Assume that it is true for $k.$ Then%
\[
\gamma_{k}=\frac{\rho_{k}-\Gamma_{k-1}^{2}}{\sqrt{q-\Gamma_{k-1}^{2}}}%
<\sqrt{\rho_{k}-\Gamma_{k-1}^{2}},
\]
i.e. $\rho_{k}>\Gamma_{k}^{2}.$ As $\rho_{k+1}>\rho_{k},$ the statement follows.

b) Evidently $\lim_{k\rightarrow\infty}\rho_{k}=q$ if (\ref{AT}) is satisfied.
The sequence $\left\{  \Gamma_{k}^{2}\right\}  $ is increasing and bounded (by
$q$). If $\zeta\overset{\mathrm{def}}{=}\lim_{k\rightarrow\infty}\Gamma
_{k}^{2}<q,$ then $\lim_{k\rightarrow\infty}\gamma_{k}=\sqrt{q-\zeta}>0,$ a
contradiction to the boundedness of $\left\{  \Gamma_{k}^{2}\right\}  .$

c) Linearization of $\psi$ around $q$ easily shows that the convergence is
exponentially fast if $\psi^{\prime}\left(  q\right)  <1.$
\end{proof}

Remark that by a) of the above lemma, one has $\gamma_{k}>0$ for all $k.$

Let $\Pi_{j}$ be the orthogonal projection in $\mathbb{R}^{N}$, with respect
to the inner product $\left\langle \cdot,\cdot\right\rangle ,$ onto
$\operatorname*{span}\left(  \mathbf{m}^{\left(  1\right)  },\ldots
,\mathbf{m}^{\left(  j\right)  }\right)  .$ We set%
\begin{equation}
\mathbf{M}^{\left(  k,j\right)  }\overset{\mathrm{def}}{=}\mathbf{m}^{\left(
k\right)  }-\Pi_{j}\left(  \mathbf{m}^{\left(  k\right)  }\right)
,\ j<k,\label{Def_Mkj}%
\end{equation}
and%
\begin{equation}
\mathbf{M}^{\left(  k\right)  }\overset{\mathrm{def}}{=}\mathbf{M}^{\left(
k,k-1\right)  }.\label{Def_Mk}%
\end{equation}

Let%
\begin{equation}
\mathbf{\phi}^{\left(  k\right)  }\overset{\mathrm{def}}{=}\frac
{\mathbf{M}^{\left(  k\right)  }}{\left\Vert \mathbf{M}^{\left(  k\right)
}\right\Vert }\label{Def_Phik}%
\end{equation}
if $\left\Vert \mathbf{M}^{\left(  k\right)  }\right\Vert \neq0.$ In case
$\mathbf{m}^{\left(  k\right)  }\in\operatorname*{span}\left(  \mathbf{m}%
^{\left(  1\right)  },\ldots,\mathbf{m}^{\left(  k-1\right)  }\right)  ,$ we
define $\mathbf{\phi}^{\left(  k\right)  }\overset{\mathrm{def}}{=}%
\mathbf{1,}$ to have it defined everywhere, but we will see that this happens
only with exponentially small probability. Remark that $\mathbf{\phi}^{\left(
1\right)  }=\mathbf{1}.$

The key result is:

\begin{proposition}
\label{Prop_Mhat_main}

For all $k\in\mathbb{N}$%
\begin{equation}
\left\Vert \mathbf{m}^{\left(  k\right)  }\right\Vert ^{2}\simeq
q,\label{Main1}%
\end{equation}
and for $1\leq j<k$%
\begin{equation}
\left\langle \mathbf{m}^{\left(  j\right)  },\mathbf{m}^{\left(  k\right)
}\right\rangle \simeq\rho_{j},\label{Main2}%
\end{equation}%
\begin{equation}
\left\langle \mathbf{\phi}^{\left(  j\right)  },\mathbf{m}^{\left(  k\right)
}\right\rangle \simeq\gamma_{j}.\label{Main3}%
\end{equation}

\end{proposition}

\begin{proof}
[Proof of Theorem \ref{Th_Main1} from Proposition \ref{Prop_Mhat_main}]As the
variables $\mathbf{m}^{\left(  k\right)  }$ are bounded, (\ref{Main1}) implies%
\[
\lim_{N\rightarrow\infty}\mathbb{E}\left\Vert \mathbf{m}^{\left(  k\right)
}\right\Vert ^{2}=q,
\]
and similarly for the other statements.%
\[
\mathbb{E}\left\Vert \mathbf{m}^{\left(  k+1\right)  }-\mathbf{m}^{\left(
k\right)  }\right\Vert ^{2}=\mathbb{E}\left\Vert \mathbf{m}^{\left(  k\right)
}\right\Vert ^{2}+\mathbb{E}\left\Vert \mathbf{m}^{\left(  k-1\right)
}\right\Vert ^{2}-2\mathbb{E}\left\langle \mathbf{m}^{\left(  k\right)
},\mathbf{m}^{\left(  k-1\right)  }\right\rangle .
\]
Taking the $N\rightarrow\infty$ limit, using Proposition \ref{Prop_Mhat_main},
this converges to $2q-2\rho_{k-1}.$ From Lemma \ref{Le_ConvergenceConstants},
the claim follows.
\end{proof}

\begin{remark}
Proposition \ref{Prop_Mhat_main} is true for all temperatures. However, beyond
the AT-line, it does not give much information on the behavior of the
$\mathbf{m}^{\left(  k\right)  }$ for large $k.$ It would be very interesting
to know if these iterates satisfy some structural properties beyond the AT-line.
\end{remark}

The main task is to prove the Proposition \ref{Prop_Mhat_main}. It follows by
an involved induction argument. We first remark that (\ref{Main3}) is a
consequence of (\ref{Main1}) and (\ref{Main2}).

If $J\in\mathbb{N}$ let $\operatorname{COND}\left(  J\right)  $ be the
statement that (\ref{Main1}) and (\ref{Main2}) hold for $k\leq J.$
$\operatorname{COND}\left(  1\right)  $ is evidently true.

$\operatorname{COND}\left(  J\right)  $ implies that for all $k\leq J,$ we
have with%
\[
\delta_{k}\overset{\mathrm{def}}{=}\sqrt{q-\Gamma_{k-1}^{2}}/2>0
\]
that%
\[
\mathbb{P}\left(  \left\Vert \mathbf{M}^{\left(  k\right)  }\right\Vert
\leq\delta_{k}\right)  \leq C_{k}\exp\left[  -N/C_{k}\right]  .
\]
If we put%
\begin{equation}
A_{J}\overset{\mathrm{def}}{=}\bigcap_{k=1}^{J}\left\{  \left\Vert
\mathbf{M}^{\left(  k\right)  }\right\Vert >\delta_{k}\right\}
,\label{AJ_Def}%
\end{equation}
then%
\begin{equation}
\mathbb{P}\left(  A_{J}\right)  \geq1-C_{J}\exp\left[  -N/C_{J}\right]
.\label{AJ_est}%
\end{equation}

Evidently, all variables $\mathbf{\phi}^{\left(  k\right)  }$ are bounded by a
constant on $A_{J},$ if $k\leq J.$ The constant may depend on $J,$ of course.
The $\mathbf{m}^{\left(  k\right)  }$ are bounded by $1$ everywhere.

\section{Iterative modifications of the interaction variables}

Let $\mathcal{G}$ be a sub-$\sigma$-field of $\mathcal{F},$ and $\mathbf{y}%
=\left(  y_{ij}\right)  _{1\leq i,j\leq N}$ be a random matrix. We are only
interested in the case where $\mathbf{y}$ is symmetric and $0$ on the
diagonal, but this is not important for the moment. We assume that
$\mathbf{y}$ is jointly Gaussian, conditioned on $\mathcal{G},$ i.e. there is
a positive semidefinite $N^{2}\times N^{2}$- $\mathcal{G}$-m.b. matrix
$\Gamma$ such that%
\[
E\left(  \left.  \exp\left[  i\sum\nolimits_{k,j}t_{kj}y_{kj}\right]
\right\vert \mathcal{G}\right)  =\exp\left[  -\frac{1}{2}\sum
\nolimits_{k,k^{\prime},j,j^{\prime}}t_{kj}\Gamma_{kj,k^{\prime}j^{\prime}%
}t_{k^{\prime}j^{\prime}}\right]  .
\]
(We do not assume that $\mathbf{y}$ is Gaussian, unconditionally). Consider a
$\mathcal{G}$-measurable random vector $\mathbf{x}$, and the linear space of
random variables%
\[
\mathcal{L}\overset{\mathrm{def}}{=}\left\{  \sum\nolimits_{i=1}^{N}%
a_{i}\left(  \mathbf{y\mathbf{x}}\right)  _{i}:a_{1},\ldots,a_{N}%
\ \mathcal{G}-\mathrm{measurable}\right\}  .
\]
We consider the linear projection $\pi_{\mathcal{L}}\left(  \mathbf{y}\right)
$ of $\mathbf{y}$ onto $\mathcal{L},$ which is defined to be the unique matrix
with components $\pi_{\mathcal{L}}\left(  y_{ij}\right)  $ in $\mathcal{L}$
which satisfy%
\[
\mathbb{E}\left(  \left\{  y_{ij}-\pi_{\mathcal{L}}\left(  y_{ij}\right)
\right\}  U%
%TCIMACRO{\U{a6}}%
%BeginExpansion
\vert
%EndExpansion
\mathcal{G}\right)  =0,\ \forall U\in\mathcal{L}.
\]
As $\mathbf{y}$ is assumed to be conditionally Gaussian, given $\mathcal{G},$
it follows that $\mathbf{y}-\pi_{\mathcal{L}}\left(  \mathbf{y}\right)  $ is
conditionally independent of the variables in $\mathcal{L},$ given
$\mathcal{G}.$

If $\mathbf{y}$ is symmetric, then clearly $\pi_{\mathcal{L}}\left(
\mathbf{y}\right)  $ is symmetric, too.

\begin{remark}
\label{Rem_Proj}If $X$ is a $\mathcal{G}$-measurable random variable then
$\mathbf{y}X$ is conditionally Gaussian as well and%
\[
\pi_{\mathcal{L}}\left(  \mathbf{y}\right)  X=\pi_{\mathcal{L}}\left(
\mathbf{y}X\right)  .
\]

\end{remark}

Remark also that%
\begin{equation}
\left(  \mathbf{y}-\pi_{\mathcal{L}}\left(  \mathbf{y}\right)  \right)
\mathbf{x=\mathbf{y\mathbf{x-}}}\pi_{\mathcal{L}}\left(  \mathbf{y\mathbf{x}%
}\right)  =0,\label{key1}%
\end{equation}
as $\mathbf{y\mathbf{x}}\in\mathcal{L}.$

Using this construction, we define a sequence $\mathbf{g}^{\left(  k\right)
},k\geq1$ of matrices, and a sequence $\left\{  \mathcal{F}_{k}\right\}  $ of
sub-$\sigma$-fields of $\mathcal{F},$ starting with $\mathbf{g}^{\left(
1\right)  }\overset{\mathrm{def}}{=}\mathbf{g,}$ and $\mathcal{F}%
_{-1}=\mathcal{F}_{0}=\left\{  \emptyset,\Omega\right\}  $). The construction
is done in such a way that

\begin{enumerate}
\item[(C1)] $\mathbf{g}^{\left(  k\right)  }$ is conditionally Gaussian, given
$\mathcal{F}_{k-1}.$

\item[(C2)] $\mathbf{m}^{\left(  k\right)  },\ \mathbf{M}^{\left(  k\right)
},$ and $\mathbf{\phi}^{\left(  k\right)  }$ are $\mathcal{F}_{k-1}$-measurable
\end{enumerate}

Using that we define%
\[
\mathbf{g}^{\left(  k+1\right)  }=\mathbf{g}^{\left(  k\right)  }%
-\pi_{\mathcal{L}_{k}}\left(  \mathbf{\mathbf{g}}^{\left(  k\right)  }\right)
,
\]
with%
\[
\mathcal{L}_{k}\overset{\mathrm{def}}{=}\left\{  \sum\nolimits_{i=1}^{N}%
a_{i}\left(  \mathbf{g}^{\left(  k\right)  }\mathbf{M}^{\left(  k\right)
}\right)  _{i}:a_{i}\ \mathcal{F}_{k-1}\mathrm{-measurable}\right\}  ,
\]
i.e. we perform the above construction with $\mathcal{G}=\mathcal{F}_{k-1}$
and $\mathbf{x}=\mathbf{M}^{\left(  k\right)  }$.

Furthermore, we define%
\[
\mathcal{F}_{k+1}\overset{\mathrm{def}}{=}\sigma\left(  \mathcal{F}%
_{k},\mathbf{\xi}^{\left(  k+1\right)  }\right)  ,
\]
where%
\[
\mathbf{\xi}^{\left(  k\right)  }\overset{\mathrm{def}}{=}\mathbf{g}^{\left(
k\right)  }\mathbf{\phi}^{\left(  k\right)  }.
\]

In order that the construction is well defined, we have to inductively prove
the properties \textbf{(C1) }and \textbf{(C2). }We actually prove a condition
which is stronger than \textbf{(C1)}:

\begin{enumerate}
\item[(C1')] Conditionally on $\mathcal{F}_{k-2},$ $\mathbf{g}^{\left(
k\right)  }$ is Gaussian, and conditionally independent of $\mathcal{F}_{k-1}.
$
\end{enumerate}

\textbf{(C1') }implies that $\mathbf{g}^{\left(  k\right)  }$ is conditionally
Gaussian, given $\mathcal{F}_{k-1},$ and the conditional law, given
$\mathcal{F}_{k-1},$ is the same as given $\mathcal{F}_{k-2}.$

\begin{proof}
[Inductive proof of (C1') and (C2)]The case $k=1$ is trivial. We first prove
(C2) for $k\geq2,$ using (C1'), (C2) up to $k-1.$ We claim that%
\begin{equation}
\mathbf{m}^{\left(  k\right)  }=\operatorname*{Th}\left(  \mathbf{g}^{\left(
k-1\right)  }\mathbf{M}^{\left(  k-1\right)  }+\mathbf{R}^{\left(  k-2\right)
}\right)  ,\label{AltRep1}%
\end{equation}
where $\mathbf{R}^{\left(  k-2\right)  }$ stands for a generic $\mathcal{F}%
_{k-2}$-measurable random variable, not necessarily the same at different occurrences.

As $\mathbf{g}^{\left(  k-1\right)  }\mathbf{M}^{\left(  k-1\right)
}=\left\Vert \mathbf{M}^{\left(  k-1\right)  }\right\Vert \mathbf{\xi
}^{\left(  k-1\right)  },$ and $\mathbf{M}^{\left(  k-1\right)  }$ is
$\mathcal{F}_{k-2}$-measurable, by the induction hypothesis, it follows from
(\ref{AltRep1}) that $\mathbf{m}^{\left(  k\right)  }$ is $\mathcal{F}_{k-1}%
$-measurable The statements for $\mathbf{M}^{\left(  k\right)  },\mathbf{\phi
}^{\left(  k\right)  }$ are then trivial consequences.

We therefore have to prove (\ref{AltRep1}). We prove by induction on $j$ that%
\begin{equation}
\mathbf{m}^{\left(  k\right)  }=\operatorname*{Th}\left(  \mathbf{g}^{\left(
j\right)  }\mathbf{M}^{\left(  k-1,j-1\right)  }+\mathbf{R}^{\left(
k-2\right)  }\right)  .\label{AltRep2}%
\end{equation}
The case $j=1$ follows from the definition of $\mathbf{m}^{\left(  k\right)
}, $ and the case $j=k-1$ is (\ref{AltRep1}).

Assume that (\ref{AltRep2}) is true for $j<k-1.$ We replace $\mathbf{g}%
^{\left(  j\right)  }$ by $\mathbf{g}^{\left(  j+1\right)  }$ through the
recursive definition%
\begin{align*}
\mathbf{m}^{\left(  k\right)  }  & =\operatorname*{Th}\left(  \mathbf{g}%
^{\left(  j+1\right)  }\mathbf{M}^{\left(  k-1,j-1\right)  }+\pi
_{\mathcal{L}_{j}}\left(  \mathbf{g}^{\left(  j\right)  }\right)
\mathbf{M}^{\left(  k-1,j-1\right)  }+\mathbf{R}^{\left(  k-2\right)  }\right)
\\
& =\operatorname*{Th}\left(  \mathbf{g}^{\left(  j+1\right)  }\mathbf{M}%
^{\left(  k-1,j-1\right)  }+\mathbf{R}^{\left(  k-2\right)  }\right)  ,
\end{align*}
as $\pi_{\mathcal{L}_{j}}\left(  \mathbf{g}^{\left(  j\right)  }\right)  $ is
$\mathcal{F}_{j}$-measurable and therefore $\pi_{\mathcal{L}_{j}}\left(
\mathbf{g}^{\left(  j\right)  }\right)  \mathbf{M}^{\left(  k-1,j-1\right)  }$
is $\mathcal{F}_{k-2}$-measurable

Using (\ref{key1}), one gets $\mathbf{g}^{\left(  j+1\right)  }\mathbf{M}%
^{\left(  j\right)  }=0,$ and therefore%
\[
\mathbf{g}^{\left(  j+1\right)  }\mathbf{M}^{\left(  k-1,j-1\right)
}=\mathbf{g}^{\left(  j+1\right)  }\mathbf{M}^{\left(  k-1,j\right)  }.
\]
This proves (\ref{AltRep1}), and therefore (C2) for $k.$

We next prove (C1') for $k.$%
\[
\mathbf{g}^{\left(  k\right)  }=\mathbf{g}^{\left(  k-1\right)  }%
-\pi_{\mathcal{L}_{k-1}}\left(  \mathbf{g}^{\left(  k-1\right)  }\right)  .
\]
We condition on $\mathcal{F}_{k-2}.$ By (C2), $\mathbf{M}^{\left(  k-1\right)
}$ is $\mathcal{F}_{k-2}$-measurable As $\mathbf{g}^{\left(  k-1\right)  },$
conditioned on $\mathcal{F}_{k-3}$, is Gaussian, and independent of
$\mathcal{F}_{k-2},$ it has the same distribution also conditioned on
$\mathcal{F}_{k-2}.$ By the construction of $\mathbf{g}^{\left(  k\right)  },$
this variable is, conditioned on $\mathcal{F}_{k-2},$ independent of
$\mathcal{F}_{k-1},$ and conditionally Gaussian.
\end{proof}

\begin{lemma}
\label{Le_Orth1}For $m<k,$ one has%
\[
\mathbf{g}^{\left(  k\right)  }\mathbf{\phi}^{\left(  m\right)  }=0.
\]

\end{lemma}

\begin{proof}
The proof is by induction on $k.$ For $k=1,$ there is nothing to prove.

Assume that the statement is proved up to $k.$ We want to prove $\mathbf{g}%
^{\left(  k+1\right)  }\mathbf{\phi}^{\left(  m\right)  }=0$ for $m\leq k.$
The case $m=k$ is covered by (\ref{key1}). For $m<k,$ it follows by Remark
\ref{Rem_Proj}, as $\mathbf{\phi}^{\left(  m\right)  }$ is $\mathcal{F}_{k-1}%
$-measurable, that%
\[
\pi_{\mathcal{L}_{k}}\left(  \mathbf{g}^{\left(  k\right)  }\right)
\mathbf{\phi}^{\left(  m\right)  }=\pi_{\mathcal{L}_{k}}\left(  \mathbf{g}%
^{\left(  k\right)  }\mathbf{\phi}^{\left(  m\right)  }\right)  ,
\]
and therefore%
\begin{align*}
\mathbf{g}^{\left(  k+1\right)  }\mathbf{\phi}^{\left(  m\right)  }  &
=\mathbf{g}^{\left(  k\right)  }\mathbf{\phi}^{\left(  m\right)  }%
-\pi_{\mathcal{L}_{k}}\left(  \mathbf{g}^{\left(  k\right)  }\right)
\mathbf{\phi}^{\left(  m\right)  }\\
& =\mathbf{g}^{\left(  k\right)  }\mathbf{\phi}^{\left(  m\right)  }%
-\pi_{\mathcal{L}_{k}}\left(  \mathbf{g}^{\left(  k\right)  }\mathbf{\phi
}^{\left(  m\right)  }\right)  =0,
\end{align*}
as $\mathbf{g}^{\left(  k\right)  }\mathbf{\phi}^{\left(  m\right)  }=0$ by
the symmetry of $\mathbf{g}^{\left(  k\right)  }$ and the induction hypothesis.
\end{proof}

\begin{lemma}
\label{Li_Orth2}If $m<k,$ then%
\[
\sum_{i}\xi_{i}^{\left(  k\right)  }\phi_{i}^{\left(  m\right)  }=0.
\]

\end{lemma}

\begin{proof}%
\begin{align*}
\sum_{i}\xi_{i}^{\left(  k\right)  }\phi_{i}^{\left(  m\right)  }  & =\sum
_{i}\sum_{j}g_{ij}^{\left(  k\right)  }\phi_{j}^{\left(  k\right)  }\phi
_{i}^{\left(  m\right)  }\\
& =\sum_{j}\phi_{j}^{\left(  k\right)  }\sum_{i}g_{ij}^{\left(  k\right)
}\phi_{i}^{\left(  m\right)  }\\
& =\sum_{j}\phi_{j}^{\left(  k\right)  }\sum_{i}g_{ji}^{\left(  k\right)
}\phi_{i}^{\left(  m\right)  }=0
\end{align*}
for $m<k,$ by the previous lemma.
\end{proof}

\section{Computation of the conditional covariances of $\mathbf{g}^{\left(
k\right)  }.$\label{Sect_Gauss}}

We introduce some more notations.

We write $O_{k}\left(  N^{-r}\right)  $ for a generic $\mathcal{F}_{k}%
$-measurable random variable $X$ which satisfies%
\[
\mathbb{P}\left(  N^{r}X\geq K\right)  \leq C\exp\left[  -N/C\right]  ,
\]
for some $K>0.$ The constants $C,K>0$ here may depend on $h,\beta,$ and the
level $k,$ and on the formula where they appear, but on nothing else, in
particular not on $N,$ and any further indices. For instance, if we write%
\[
X_{ij}=Y_{ij}+O_{k}\left(  N^{-5}\right)  ,
\]
we mean that there exists $C\left(  \beta,h,k\right)  ,\ K\left(
\beta,h,k\right)  >0$ with%
\[
\sup_{ij}\mathbb{P}\left(  N^{5}\left\vert X_{ij}-Y_{ij}\right\vert \geq
K\right)  \leq C\exp\left[  -N/C\right]  .
\]
Furthermore, in such a case, it is tacitly assumed that $X_{ij}-Y_{ij}$ are
$\mathcal{F}_{k}$-measurable

Evidently, if $X,Y$ are $O_{k}\left(  N^{-r}\right)  ,$ then $X+Y$ is
$O_{k}\left(  N^{-r}\right)  ,$ and if $X$ is $O_{k}\left(  N^{-r}\right)  ,$
and $Y$ is $O_{k}\left(  N^{-s}\right)  ,$ then $XY$ is $O_{k}\left(
N^{-r-s}\right)  .$

We write $\mathbb{E}_{k}$ for the conditional expectation, given
$\mathcal{F}_{k}.$

We will finally prove the validity of the following relations:%
\begin{equation}
\mathbb{E}_{k-2}g_{ij}^{\left(  k\right)  2}=\frac{1}{N}+O_{k-2}\left(
N^{-2}\right)  .\label{GaussMain1}%
\end{equation}%
\begin{equation}
\mathbb{E}_{k-2}g_{ij}^{\left(  k\right)  }g_{jt}^{\left(  k\right)  }%
=-\sum_{m=1}^{k-1}\frac{\phi_{i}^{\left(  m\right)  }\phi_{t}^{\left(
m\right)  }}{N^{2}}+O_{k-2}\left(  N^{-3}\right)  ,\ \forall t\neq
i,j\label{GaussMain2}%
\end{equation}%
\begin{equation}
\mathbb{E}_{k-2}g_{ij}^{\left(  k\right)  }g_{st}^{\left(  k\right)  }%
=\frac{\alpha_{ijst}^{\left(  k\right)  }}{N^{3}}+O_{k-2}\left(
N^{-4}\right)  ,\ \mathrm{if\ }\left\{  s,t\right\}  \cap\left\{  i,j\right\}
=\emptyset\label{GaussMain3}%
\end{equation}
where%
\[
\alpha_{ijst}^{\left(  k\right)  }=\sum_{m=1}^{k-1}\sum_{A\subset\left\{
i,j,s,t\right\}  }\lambda_{m,A}^{\left(  k\right)  }\phi_{A}^{\left(
m\right)  }%
\]
with%
\[
\phi_{A}^{\left(  m\right)  }\overset{\mathrm{def}}{=}\prod_{u\in A}\phi
_{u}^{\left(  m\right)  }.
\]
The $\lambda_{m,A}^{\left(  k\right)  }$ are real numbers, not random
variables, which depend on $A$ only through the type of subset which is taken.
For instance, there is only one number (for every $m,k$) if all four indices
are taken.

The main result of this section is:

\begin{proposition}
\label{Th_Gauss_Main}Let $J\in\mathbb{N}$, assume $\operatorname*{COND}\left(
J\right)  ,$ and assume the validity of (\ref{GaussMain1}) - (\ref{GaussMain3}%
) hold for $k\leq J.$ Then they hold for $k=J+1.$
\end{proposition}

The main point with assuming $\operatorname*{COND}\left(  J\right)  $ is
(\ref{AJ_est}). On $A_{J},$ the variables $\mathbf{\phi}^{\left(  k\right)  }$
are bounded for $k\leq J.$

\begin{lemma}
\label{Le_G1}Assume (\ref{GaussMain1}) - (\ref{GaussMain3}) for $k=J,$ and
(\ref{AJ_est})$.$ Then

\begin{enumerate}
\item[a)]
\begin{equation}
\mathbb{E}_{J-1}\xi_{i}^{\left(  J\right)  2}=1+O_{J-1}\left(  N^{-1}\right)
.\label{var_xi}%
\end{equation}

\item[b)]
\begin{equation}
\mathbb{E}_{J-1}\xi_{i}^{\left(  J\right)  }\xi_{j}^{\left(  J\right)  }%
=\frac{1}{N}\phi_{i}^{\left(  J\right)  }\phi_{j}^{\left(  J\right)  }%
-\frac{1}{N}\sum_{r=1}^{J-1}\phi_{i}^{\left(  r\right)  }\phi_{j}^{\left(
r\right)  }+O_{J-1}\left(  N^{-2}\right) \label{Corr_xi}%
\end{equation}

\item[c)]
\begin{equation}
\mathbb{E}_{J-1}g_{ij}^{\left(  J\right)  }\xi_{i}^{\left(  J\right)  }%
=\frac{\phi_{j}^{\left(  J\right)  }}{N}+O_{J-1}\left(  N^{-2}\right)
.\label{corr_xi_g_diag}%
\end{equation}

\item[d)] For $s\neq i,j$%
\begin{equation}
\mathbb{E}_{J-1}g_{ij}^{\left(  J\right)  }\xi_{s}^{\left(  J\right)  }%
=-\frac{\phi_{i}^{\left(  J\right)  }}{N^{2}}\sum_{m=1}^{J-1}\phi_{j}^{\left(
m\right)  }\phi_{s}^{\left(  m\right)  }-\frac{\phi_{j}^{\left(  J\right)  }%
}{N^{2}}\sum_{m=1}^{J-1}\phi_{i}^{\left(  m\right)  }\phi_{s}^{\left(
m\right)  }+O_{J-1}\left(  N^{-3}\right)  .\label{corr_xi_g_offdiag}%
\end{equation}

\end{enumerate}
\end{lemma}

\begin{proof}
a) As $\mathbf{\phi}^{\left(  J\right)  }$ is $\mathcal{F}_{J-1}$-measurable,
and $\mathbf{g}^{\left(  J\right)  }$ is independent of $\mathcal{F}_{J-1},$
conditionally on $\mathcal{F}_{J-2},$ we get%
\begin{align*}
\mathbb{E}_{J-1}\xi_{i}^{\left(  J\right)  2}  & =\sum_{s,t\neq i}\phi
_{s}^{\left(  J\right)  }\phi_{t}^{\left(  J\right)  }\mathbb{E}_{J-1}\left(
g_{is}^{\left(  J\right)  }g_{it}^{\left(  J\right)  }\right)  =\sum_{s,t\neq
i}\phi_{s}^{\left(  J\right)  }\phi_{t}^{\left(  J\right)  }\mathbb{E}%
_{J-2}\left(  g_{is}^{\left(  J\right)  }g_{it}^{\left(  J\right)  }\right) \\
& =\sum_{s\neq i}\phi_{s}^{\left(  J\right)  2}\mathbb{E}_{J-2}\left(
g_{is}^{\left(  J\right)  2}\right)  +\sum_{\substack{s,t\neq i \\s\neq
t}}\phi_{s}^{\left(  J\right)  }\phi_{t}^{\left(  J\right)  }\mathbb{E}%
_{J-2}\left(  g_{is}^{\left(  J\right)  }g_{it}^{\left(  J\right)  }\right)  .
\end{align*}
Using (\ref{GaussMain1}), (\ref{GaussMain2}), and the boundedness of the
$\phi$'s on $A_{J}$, and $N^{-1}\sum_{i}\phi_{i}^{\left(  J\right)
2}=1,\ \sum_{i}\phi_{i}^{\left(  J\right)  }\phi_{i}^{\left(  m\right)  }=0$
for $m<J$, we get%
\[
\mathbb{E}_{J-1}\xi_{i}^{\left(  J\right)  2}=1+O_{J-1}\left(  N^{-1}\right)
.
\]

b)%
\[
\mathbb{E}_{J-1}\xi_{i}^{\left(  J\right)  }\xi_{j}^{\left(  J\right)  }%
=\sum_{s\neq i,t\neq j}\phi_{s}^{\left(  J\right)  }\phi_{t}^{\left(
J\right)  }\mathbb{E}_{J-2}\left(  g_{is}^{\left(  J\right)  }g_{jt}^{\left(
J\right)  }\right)
\]
We split the sum over $\left(  s,t\right)  $ into the one summand $s=j,t=i$,
in $A=\left\{  \left(  s,s\right)  :s\neq i,j\right\}  ,\ B=\left\{  \left(
j,t\right)  :t\neq i,j\right\}  ,\ C=\left\{  \left(  s,i\right)  :s\neq
i,j\right\}  ,$ and $D=\left\{  \left(  s,t\right)  :\left\{  s,t\right\}
\cap\left\{  i,j\right\}  =\emptyset\right\}  .$ The one summand $s=j,t=i$
gives $\phi_{i}^{\left(  J\right)  }\phi_{j}^{\left(  J\right)  }%
/N+O_{J-1}\left(  N^{-2}\right)  .$%
\begin{align*}
\sum_{A}  & =\sum_{s\neq i,j}\phi_{s}^{\left(  J\right)  2}\mathbb{E}%
_{J-2}\left(  g_{is}^{\left(  J\right)  }g_{js}^{\left(  J\right)  }\right)
=\sum_{s\neq i,j}\phi_{s}^{\left(  J\right)  2}\left\{  -\sum_{m=1}^{J-1}%
\frac{\phi_{i}^{\left(  m\right)  }\phi_{j}^{\left(  m\right)  }}{N^{2}%
}+O_{J-2}\left(  N^{-3}\right)  \right\} \\
& =-\sum_{m=1}^{J-1}\frac{\phi_{i}^{\left(  m\right)  }\phi_{j}^{\left(
m\right)  }}{N}+O_{J-1}\left(  N^{-2}\right)  .
\end{align*}%
\begin{align*}
\sum_{B}  & =\sum_{t\neq i,j}\phi_{j}^{\left(  J\right)  }\phi_{t}^{\left(
J\right)  }\mathbb{E}_{J-2}\left(  g_{ij}^{\left(  J\right)  }g_{jt}^{\left(
J\right)  }\right) \\
& =\sum_{t\neq i,j}\phi_{j}^{\left(  J\right)  }\phi_{t}^{\left(  J\right)
}\left\{  -\sum_{m=1}^{J-1}\frac{\phi_{i}^{\left(  m\right)  }\phi
_{t}^{\left(  m\right)  }}{N^{2}}+O_{J-2}\left(  N^{-3}\right)  \right\}  .
\end{align*}
Because $\left\langle \mathbf{\phi}^{\left(  J\right)  },\mathbf{\phi
}^{\left(  m\right)  }\right\rangle =0$ for $m<J,$ this is seen to be
$O_{J-1}\left(  N^{-2}\right)  .$ The same applies to $\sum_{C}.$

It remains to consider the last part $\sum_{D}.$ Here we have to use the
expression for $\mathbb{E}_{J-2}\left(  g_{ij}^{\left(  J\right)  }%
g_{st}^{\left(  J\right)  }\right)  $ where $\left\{  i,j\right\}
\cap\left\{  s,t\right\}  =\emptyset$ given in c) of Theorem
\ref{Th_Gauss_Main}.%
\begin{align*}
& \sum_{s,t:\left\{  s,t\right\}  \cap\left\{  i,j\right\}  =\emptyset}%
\phi_{s}^{\left(  J\right)  }\phi_{t}^{\left(  J\right)  }\left[  \frac
{1}{N^{3}}\sum_{m=1}^{J-1}\sum_{A\subset\left\{  i,j,s,t\right\}  }%
\lambda_{m,A}^{\left(  J\right)  }\phi_{A}^{\left(  m\right)  }+O_{J-2}\left(
N^{-4}\right)  \right] \\
& =\frac{1}{N^{3}}\sum_{s,t:\left\{  s,t\right\}  \cap\left\{  i,j\right\}
=\emptyset}\phi_{s}^{\left(  J\right)  }\phi_{t}^{\left(  J\right)  }%
\sum_{m=1}^{J-1}\sum_{A\subset\left\{  i,j,s,t\right\}  }\lambda
_{m,A}^{\left(  J\right)  }\phi_{A}^{\left(  m\right)  }+O_{J-2}\left(
N^{-2}\right)  .
\end{align*}
Take e.g. $A=\left\{  i,j,s\right\}  .$ Then $\lambda_{m,A}^{\left(  J\right)
}=\lambda_{m,3}^{\left(  J\right)  }$ with no further dependence of this
number on $i,j,s.$ So we get for this part for any summand on $m$ with $m<J$%
\[
\frac{1}{N^{3}}\lambda_{m,3}^{\left(  J\right)  }\sum_{s,t:\left\{
s,t\right\}  \cap\left\{  i,j\right\}  =\emptyset}\phi_{s}^{\left(  J\right)
}\phi_{t}^{\left(  J\right)  }\phi_{s}^{\left(  m\right)  }\phi_{t}^{\left(
m\right)  }\phi_{i}^{\left(  m\right)  }.
\]
Using again $\left\langle \mathbf{\phi}^{\left(  J\right)  },\mathbf{\phi
}^{\left(  m\right)  }\right\rangle =0,$ we get that this is $O_{J-1}\left(
N^{-2}\right)  .$ This applies in the same way to all the parts. Therefore b) follows.

c)%
\begin{align*}
\mathbb{E}_{J-1}g_{ij}^{\left(  J\right)  }\xi_{i}^{\left(  J\right)  }  &
=\sum_{t\neq i}\phi_{t}^{\left(  J\right)  }\mathbb{E}_{J-2}\left(
g_{ij}^{\left(  J\right)  }g_{it}^{\left(  J\right)  }\right)  =\frac{\phi
_{j}^{\left(  J\right)  }}{N}+O_{J-1}\left(  N^{-2}\right) \\
& +\sum_{t\neq i,j}\phi_{t}^{\left(  J\right)  }\left[  -\sum_{m=1}^{J-1}%
\frac{\phi_{j}^{\left(  m\right)  }\phi_{t}^{\left(  m\right)  }}{N^{2}%
}\right]  +O_{J-1}\left(  N^{-2}\right) \\
& =\frac{\phi_{j}^{\left(  J\right)  }}{N}+O_{J-1}\left(  N^{-2}\right)  ,
\end{align*}
due to the orthogonality of the $\mathbf{\phi}^{\left(  m\right)  }.$

d)%
\begin{align*}
\mathbb{E}g_{ij}^{\left(  J\right)  }\xi_{s}^{\left(  J\right)  }  &
=\sum_{t\neq s}\phi_{t}^{\left(  J\right)  }\mathbb{E}g_{ij}^{\left(
J\right)  }g_{st}^{\left(  J\right)  }\\
& =\phi_{i}^{\left(  J\right)  }\mathbb{E}g_{ij}^{\left(  J\right)  }%
g_{si}^{\left(  J\right)  }+\phi_{j}^{\left(  J\right)  }\mathbb{E}%
g_{ij}^{\left(  J\right)  }g_{sj}^{\left(  J\right)  }+O_{J-1}\left(
N^{-3}\right)  ,
\end{align*}
due again to (\ref{GaussMain3}). We therefore get%
\[
\mathbb{E}g_{ij}^{\left(  J\right)  }\xi_{s}^{\left(  J\right)  }=-\frac
{1}{N^{2}}\sum_{m=1}^{J-1}\phi_{s}^{\left(  m\right)  }\left[  \phi
_{i}^{\left(  J\right)  }\phi_{j}^{\left(  m\right)  }+\phi_{j}^{\left(
J\right)  }\phi_{i}^{\left(  m\right)  }\right]  +O_{J-1}\left(
N^{-3}\right)
\]

\end{proof}

\begin{lemma}
\label{Le_G2}We assume the same as in Lemma \ref{Le_G1}. Put%
\[
\hat{g}_{ij}^{\left(  J\right)  }\overset{\mathrm{def}}{=}g_{ij}^{\left(
J\right)  }-\frac{\phi_{i}^{\left(  J\right)  }\xi_{j}^{\left(  J\right)
}+\phi_{j}^{\left(  J\right)  }\xi_{i}^{\left(  J\right)  }}{N}+\phi
_{i}^{\left(  J\right)  }\phi_{j}^{\left(  J\right)  }\frac{1}{N^{2}}%
\sum_{r=1}^{N}\phi_{r}^{\left(  J\right)  }\xi_{r}^{\left(  J\right)  }.
\]
Then%
\begin{equation}
g_{ij}^{\left(  J+1\right)  }=\hat{g}_{ij}^{\left(  J\right)  }-\sum
_{s}x_{ij,s}^{\left(  J\right)  }\xi_{s}^{\left(  J\right)  }%
,\label{final_correction}%
\end{equation}
where the $\mathcal{F}_{J-1}$-measurable coefficients $x_{ij,s}^{\left(
J\right)  }$ satisfy%
\begin{equation}
\sum_{s}x_{ij,s}^{\left(  J\right)  }\phi_{s}^{\left(  m\right)  }=0,\ \forall
i,j,\ \forall m<J,\label{Ortho1}%
\end{equation}
with%
\begin{align*}
x_{ij,s}^{\left(  J\right)  }  & =O_{J-1}\left(  N^{-2}\right)  ,\ s\in
\left\{  i,j\right\}  ,\\
x_{ij,s}^{\left(  J\right)  }  & =O_{J-1}\left(  N^{-3}\right)  ,\ s\notin
\left\{  i,j\right\}  ,
\end{align*}

\end{lemma}

\begin{proof}
The existence of $\mathcal{F}_{J-1}$-measurable coefficients $x_{ij,s}%
^{\left(  J\right)  }$ comes from linear algebra.

Remark that%
\[
\sum_{s}\xi_{s}^{\left(  J\right)  }\phi_{s}^{\left(  m\right)  }=\sum
_{s,j}\phi_{s}^{\left(  m\right)  }g_{sj}^{\left(  J\right)  }\phi
_{j}^{\left(  J\right)  }=\sum_{j}\phi_{j}^{\left(  J\right)  }\left[
\sum\nolimits_{s}g_{js}^{\left(  J\right)  }\phi_{s}^{\left(  m\right)
}\right]  =0.
\]
Therefore, we can replace the $x_{ij,\cdot}^{\left(  J\right)  }$ by%
\[
x_{ij,\cdot}^{\left(  J\right)  }-\sum_{m=1}^{J-1}\left\langle x_{ij,\cdot
}^{\left(  J\right)  },\mathbf{\phi}^{\left(  m\right)  }\right\rangle
\mathbf{\phi}^{\left(  m\right)  }%
\]
which satisfy the desired property (\ref{Ortho1}).

We keep $i,j$ fixed for the moment and write $x_{s}$ for $x_{ij,s}^{\left(
J\right)  }.$ The requirement for them is that for all $t$%
\[
\mathbb{E}_{J-1}\left(  \left(  \hat{g}_{ij}^{\left(  J\right)  }%
-\sum\nolimits_{s}x_{s}\xi_{s}^{\left(  J\right)  }\right)  \xi_{t}^{\left(
J\right)  }\right)  =0.
\]

From Lemma \ref{Le_G2}, we get%
\[
\mathbb{E}_{J-1}\left(  \hat{g}_{ij}^{\left(  J\right)  }\xi_{i}^{\left(
J\right)  }\right)  =O_{J-1}\left(  N^{-2}\right)  ,
\]
and the same for $\mathbb{E}_{J-1}\left(  \hat{g}_{ij}^{\left(  J\right)  }%
\xi_{j}^{\left(  J\right)  }\right)  .$ For $t\notin\left\{  i,j\right\}  ,$
we have%
\begin{align*}
\mathbb{E}_{J-1}\left(  \hat{g}_{ij}^{\left(  J\right)  }\xi_{t}^{\left(
J\right)  }\right)   & =-\frac{\phi_{i}^{\left(  J\right)  }}{N^{2}}\sum
_{m=1}^{J-1}\phi_{j}^{\left(  m\right)  }\phi_{t}^{\left(  m\right)  }%
-\frac{\phi_{j}^{\left(  J\right)  }}{N^{2}}\sum_{m=1}^{J-1}\phi_{i}^{\left(
m\right)  }\phi_{t}^{\left(  m\right)  }+O_{J-1}\left(  N^{-3}\right) \\
& -\frac{\phi_{i}^{\left(  J\right)  }}{N}\left\{  \frac{1}{N}\phi
_{j}^{\left(  J\right)  }\phi_{t}^{\left(  J\right)  }-\frac{1}{N}\sum
_{m=1}^{J-1}\phi_{j}^{\left(  m\right)  }\phi_{t}^{\left(  m\right)  }%
+O_{J-1}\left(  N^{-2}\right)  \right\} \\
& -\frac{\phi_{j}^{\left(  J\right)  }}{N}\left\{  \frac{1}{N}\phi
_{i}^{\left(  J\right)  }\phi_{t}^{\left(  J\right)  }-\frac{1}{N}\sum
_{m=1}^{J-1}\phi_{i}^{\left(  m\right)  }\phi_{t}^{\left(  m\right)  }%
+O_{J-1}\left(  N^{-2}\right)  \right\} \\
& +\phi_{i}^{\left(  J\right)  }\phi_{j}^{\left(  J\right)  }\frac{1}{N^{2}%
}\sum_{r}\phi_{r}^{\left(  J\right)  }\mathbb{E}_{J-1}\xi_{r}^{\left(
J\right)  }\xi_{t}^{\left(  J\right)  }\\
& =-\frac{2}{N^{2}}\phi_{i}^{\left(  J\right)  }\phi_{j}^{\left(  J\right)
}\phi_{t}^{\left(  J\right)  }+\frac{1}{N^{2}}\phi_{i}^{\left(  J\right)
}\phi_{j}^{\left(  J\right)  }\phi_{t}^{\left(  J\right)  }\\
& +\frac{1}{N^{2}}\phi_{i}^{\left(  J\right)  }\phi_{j}^{\left(  J\right)
}\sum_{r\neq t}\phi_{r}^{\left(  J\right)  }\left\{  \frac{1}{N}\phi
_{r}^{\left(  J\right)  }\phi_{t}^{\left(  J\right)  }-\frac{1}{N}\sum
_{m=1}^{J-1}\phi_{r}^{\left(  m\right)  }\phi_{t}^{\left(  m\right)  }\right\}
\\
& +O_{J-1}\left(  N^{-3}\right)  .
\end{align*}
Due to the orthonormality of the $\mathbf{\phi},$ one gets%
\begin{align*}
\frac{1}{N}\sum_{r\neq t}\phi_{r}^{\left(  J\right)  2}  & =1+O_{J-1}\left(
N^{-1}\right)  ,\\
\sum_{r\neq t}\phi_{r}^{\left(  J\right)  }\phi_{r}^{\left(  m\right)  }  &
=O_{J-1}\left(  N^{-1}\right)  .
\end{align*}
So we get%
\[
\mathbb{E}_{J-1}\left(  \hat{g}_{ij}^{\left(  J\right)  }\xi_{t}^{\left(
J\right)  }\right)  =O_{J-1}\left(  N^{-3}\right)  .
\]
We write for the moment $y_{t}\overset{\mathrm{def}}{=}\mathbb{E}_{J-1}\left(
\hat{g}_{ij}^{\left(  J\right)  }\xi_{t}^{\left(  J\right)  }\right)  .$ The
equations for $\left\{  x_{s}\right\}  $ are%
\[
\sum_{s}x_{s}\mathbb{E}_{J-1}\xi_{s}^{\left(  J\right)  }\xi_{t}^{\left(
J\right)  }=y_{t},\ \forall t.
\]
Writing $r_{ij}$ for the $O_{J-1}\left(  N^{-2}\right)  $ error term in
(\ref{Corr_xi}), and for $j=i,$ the $O_{J-1}\left(  N^{-1}\right)  $ error
term in (\ref{var_xi}), we arrive at%
\[
\sum_{s\neq t}x_{s}\left\{  \frac{1}{N}\phi_{s}^{\left(  J\right)  }\phi
_{t}^{\left(  J\right)  }-\frac{1}{N}\sum_{m=1}^{J-1}\phi_{s}^{\left(
m\right)  }\phi_{t}^{\left(  m\right)  }+r_{st}\right\}  +x_{t}\left(
1+r_{tt}\right)  =y_{t}.
\]
In the first summand, we sum now over all $s,$ remarking that we have assumed
that $\sum_{s}x_{s}\phi_{s}^{\left(  m\right)  }=0$ for $m<J.$ The error for
not summing over the single $t$ can be incorporated into $r_{tt}.$ We
therefore arrive at%
\[
x_{t}+\phi_{t}^{\left(  J\right)  }\frac{1}{N}\sum_{s}x_{s}\phi_{s}^{\left(
J\right)  }+\sum_{s}x_{s}r_{st}=y_{t}.
\]
Write $\Phi$ for the matrix $\left(  N^{-1}\phi_{i}^{\left(  J\right)  }%
\phi_{j}^{\left(  J\right)  }\right)  $ and $R$ for $\left(  r_{ij}\right)  .$
Then we have to invert the matrix $\left(  I+\Phi+R\right)  .$ Remark that
$\left(  I+\Phi\right)  ^{-1}=I-\Phi/2.$ Therefore%
\[
\left(  I-\Phi/2\right)  \left(  I+\Phi+R\right)  =I+\left(  I-\Phi/2\right)
R.
\]
The right hand side, we can develop as a Neumann series:%
\begin{align*}
\left(  I+\Phi+R\right)  ^{-1}\left(  I+\Phi\right)   & =\left(  I+\left(
I-\Phi/2\right)  R\right)  ^{-1}\\
& =I-\left(  I-\Phi/2\right)  R+\left[  \left(  I-\Phi/2\right)  R\right]
^{2}-\cdots
\end{align*}

\[
\left(  I+\Phi+R\right)  ^{-1}=I-\frac{\Phi}{2}-\left(  I-\frac{\Phi}%
{2}\right)  R\left(  I-\frac{\Phi}{2}\right)  +\cdots.
\]
As $\left(  \Phi\mathbf{y}\right)  _{i}=O_{J-1}\left(  N^{-3}\right)  ,$ we
get the desired conclusion.
\end{proof}

\begin{proof}
[Proof of Proposition \ref{Th_Gauss_Main}]%
\begin{align}
\mathbb{E}_{J-1}\left(  g_{ij}^{\left(  J+1\right)  }g_{st}^{\left(
J+1\right)  }\right)   & =\mathbb{E}_{J-1}\left(  \hat{g}_{ij}^{\left(
J+1\right)  }\hat{g}_{st}^{\left(  J+1\right)  }\right)  +\sum_{u}%
x_{st,u}^{\left(  J\right)  }\mathbb{E}_{J-1}\left(  \xi_{u}^{\left(
J\right)  }\hat{g}_{ij}^{\left(  J\right)  }\right) \nonumber\\
& +\sum_{u}x_{ij,u}^{\left(  J\right)  }\mathbb{E}_{J-1}\left(  \xi
_{u}^{\left(  J\right)  }\hat{g}_{st}^{\left(  J\right)  }\right)
\label{Ind_split}\\
& +\sum_{u,v}x_{ij,u}^{\left(  J\right)  }x_{st,v}^{\left(  J\right)
}\mathbb{E}_{J-1}\left(  \xi_{u}^{\left(  J\right)  }\xi_{v}^{\left(
J\right)  }\right)  .\nonumber
\end{align}
The summands involving the $x^{\left(  J\right)  }$ all only give
contributions which enter the $O_{J-1}$-terms. Take for instance $s=j,\ t\neq
i,j.$ In that case, the claimed $O_{J-1}$-term is $O_{J-1}\left(
N^{-3}\right)  .$ In the last summand of (\ref{Ind_split}), there is one
summand, namely $u=v=j,$ where the $x^{\left(  J\right)  }$ are $O_{J-1}%
\left(  N^{-2}\right)  ,$ so this summand is only $O_{J-1}\left(
N^{-4}\right)  .$%
\begin{align*}
\sum_{u}x_{jt,u}^{\left(  J\right)  }\mathbb{E}_{J-1}\left(  \xi_{u}^{\left(
J\right)  }\hat{g}_{ij}^{\left(  J\right)  }\right)   & =x_{jt,i}^{\left(
J\right)  }\mathbb{E}_{J-1}\left(  \xi_{i}^{\left(  J\right)  }\hat{g}%
_{ij}^{\left(  J\right)  }\right)  +x_{jt,j}^{\left(  J\right)  }%
\mathbb{E}_{J-1}\left(  \xi_{j}^{\left(  J\right)  }\hat{g}_{ij}^{\left(
J\right)  }\right) \\
& +x_{jt,t}^{\left(  J\right)  }\mathbb{E}_{J-1}\left(  \xi_{t}^{\left(
J\right)  }\hat{g}_{ij}^{\left(  J\right)  }\right) \\
& +\sum_{u\neq i,j,t}x_{jt,u}^{\left(  J\right)  }\mathbb{E}_{J-1}\left(
\xi_{u}^{\left(  J\right)  }\hat{g}_{ij}^{\left(  J\right)  }\right)  .
\end{align*}
From Lemma \ref{Le_G1}, we get $\mathbb{E}_{J-1}\left(  \xi_{u}^{\left(
J\right)  }\hat{g}_{ij}^{\left(  J\right)  }\right)  =O_{J-1}\left(
N^{-1}\right)  $ for $u\in\left\{  i,j\right\}  ,$ and $O_{J-1}\left(
N^{-2}\right)  $ otherwise. So the above sum gives%
\begin{align*}
& O_{J-1}\left(  N^{-3}\right)  O_{J-1}\left(  N^{-1}\right)  +O_{J-1}\left(
N^{-2}\right)  O_{J-1}\left(  N^{-1}\right) \\
& +O_{J-1}\left(  N^{-2}\right)  O_{J-1}\left(  N^{-2}\right)  +NO_{J-1}%
\left(  N^{-3}\right)  O_{J-1}\left(  N^{-2}\right) \\
& =O_{J-1}\left(  N^{-3}\right)  .
\end{align*}
The other summands behave similarly. The third and fourth summand in
(\ref{Ind_split}) behave similarly.

As another case, take $\left\{  i,j\right\}  \cap\left\{  s,t\right\}
=\emptyset,$ where we have to get $O_{J-1}\left(  N^{-4}\right)  $ for the
second to fourth summand in (\ref{Ind_split}).%
\begin{align*}
\sum_{u}x_{st,u}^{\left(  J\right)  }\mathbb{E}_{J-1}\left(  \xi_{u}^{\left(
J\right)  }\hat{g}_{ij}^{\left(  J\right)  }\right)   & =\sum_{u=i,j}%
+\sum_{u=s,t}+\sum_{u\notin\left\{  i,j,s,t\right\}  }\\
& =O_{J-1}\left(  N^{-3}\right)  O_{J-1}\left(  N^{-1}\right)  +O_{J-1}\left(
N^{-2}\right)  O_{J-1}\left(  N^{-2}\right) \\
& +NO_{J-1}\left(  N^{-3}\right)  O_{J-1}\left(  N^{-2}\right) \\
& =O_{J-1}\left(  N^{-4}\right)  .
\end{align*}%
\begin{align*}
\sum_{u,v}x_{ij,u}^{\left(  J\right)  }x_{st,v}^{\left(  J\right)  }%
\mathbb{E}_{J-1}\left(  \xi_{u}^{\left(  J\right)  }\xi_{v}^{\left(  J\right)
}\right)   & =\sum_{u=v\in\left\{  i,j,s,t\right\}  }+\sum_{u=v\notin\left\{
i,j,s,t\right\}  }+\sum_{u\in\left\{  i,j\right\}  }\sum_{v\in\left\{
s,t\right\}  }\\
& +\sum_{u\in\left\{  i,j\right\}  }\sum_{v\notin\left\{  s,t\right\}  ,\neq
u}+\sum_{v\in\left\{  s,t\right\}  }\sum_{u\notin\left\{  i,j\right\}  ,\neq
v}\\
& +\sum_{u\neq v}\sum_{u\notin\left\{  i,j\right\}  }\sum_{v\notin\left\{
s,t\right\}  }\\
& =O_{J-1}\left(  N^{-5}\right)  +NO_{J-1}\left(  N^{-6}\right) \\
& +O_{J-1}\left(  N^{-2}\right)  O_{J-1}\left(  N^{-2}\right)  O_{J-1}\left(
N^{-1}\right) \\
& +NO_{J-1}\left(  N^{-2}\right)  O_{J-1}\left(  N^{-3}\right)  O_{J-1}\left(
N^{-1}\right) \\
& +NO_{J-1}\left(  N^{-2}\right)  O_{J-1}\left(  N^{-3}\right)  O_{J-1}\left(
N^{-1}\right) \\
& +N^{2}O_{J-1}\left(  N^{-3}\right)  O_{J-1}\left(  N^{-3}\right)
O_{J-1}\left(  N^{-1}\right) \\
& =O_{J-1}\left(  N^{-5}\right)  ,
\end{align*}
which is better than required.

It therefore remains to investigate $\mathbb{E}_{J-1}\left(  \hat{g}%
_{ij}^{\left(  J+1\right)  }\hat{g}_{st}^{\left(  J+1\right)  }\right)  .$

a)%
\[
\mathbb{E}_{J-1}\left(  \hat{g}_{ij}^{\left(  J\right)  2}\right)
=\mathbb{E}_{J-1}\left[  \left(  g_{ij}^{\left(  J\right)  }-\frac{\phi
_{i}^{\left(  J\right)  }\xi_{j}^{\left(  J\right)  }+\phi_{j}^{\left(
J\right)  }\xi_{i}^{\left(  J\right)  }}{N}+\frac{\phi_{i}^{\left(  J\right)
}\phi_{j}^{\left(  J\right)  }}{N^{2}}\sum_{t}\phi_{t}^{\left(  J\right)  }%
\xi_{t}^{\left(  J\right)  }\right)  ^{2}\right]  .
\]
Using Lemma \ref{Le_G1}, one easily gets that anything except $\mathbb{E}%
_{J-1}\left(  g_{ij}^{\left(  J\right)  2}\right)  $ is $O_{J-1}\left(
N^{-2}\right)  .$ $\mathbb{E}_{J-1}\left(  g_{ij}^{\left(  J\right)
2}\right)  =\mathbb{E}_{J-2}\left(  g_{ij}^{\left(  J\right)  2}\right)  $
from the conditional independence of $\mathbf{g}^{\left(  J\right)  }$ of
$\mathcal{F}_{J-1},$ given $\mathcal{F}_{J-2}.$ So the claim follows.

b)%
\begin{align*}
\mathbb{E}_{J-1}\left(  \hat{g}_{ij}^{\left(  J\right)  }\hat{g}_{jt}^{\left(
J\right)  }\right)   & =\mathbb{E}_{J-1}\Bigg[\left(  g_{ij}^{\left(
J\right)  }-\frac{\phi_{i}^{\left(  J\right)  }\xi_{j}^{\left(  J\right)
}+\phi_{j}^{\left(  J\right)  }\xi_{i}^{\left(  J\right)  }}{N}+\frac{\phi
_{i}^{\left(  J\right)  }\phi_{j}^{\left(  J\right)  }}{N^{2}}\sum_{u}\phi
_{u}^{\left(  J\right)  }\xi_{u}^{\left(  J\right)  }\right) \\
& \times\left(  g_{jt}^{\left(  J\right)  }-\frac{\phi_{j}^{\left(  J\right)
}\xi_{t}^{\left(  J\right)  }+\phi_{t}^{\left(  J\right)  }\xi_{j}^{\left(
J\right)  }}{N}+\frac{\phi_{j}^{\left(  J\right)  }\phi_{t}^{\left(  J\right)
}}{N^{2}}\sum_{u}\phi_{u}^{\left(  J\right)  }\xi_{u}^{\left(  J\right)
}\right)  \Bigg].
\end{align*}

We write $m\times n$ for the summand, we get by multiplying the $m$-th summand
in the first bracket with the $n$-th in the second. By induction hypothesis,
we get%
\[
1\times1=-\sum_{m=1}^{J-1}\frac{\phi_{i}^{\left(  m\right)  }\phi_{t}^{\left(
m\right)  }}{N^{2}}+O_{J-2}\left(  N^{-3}\right)  .
\]
In the $1\times2$-term, only the multiplication of $g_{ij}^{\left(  J\right)
}$ with $\xi_{j}^{\left(  J\right)  }$ counts, the other part giving
$O_{J-2}\left(  N^{-3}\right)  .$ Therefore%
\begin{align*}
1\times2  & =-\frac{\phi_{t}^{\left(  J\right)  }}{N}\mathbb{E}_{J-1}%
g_{ij}^{\left(  J\right)  }\xi_{j}^{\left(  J\right)  }+O_{J-1}\left(
N^{-3}\right) \\
& =-\frac{\phi_{i}^{\left(  J\right)  }\phi_{t}^{\left(  J\right)  }}{N^{2}%
}+O_{J-1}\left(  N^{-3}\right)  .
\end{align*}
$2\times1$ gives the same. In $2\times2,$ again only the matching of $\xi
_{j}^{\left(  J\right)  }$ with $\xi_{j}^{\left(  J\right)  }$ counts, so we
get%
\[
2\times s=\frac{\phi_{i}^{\left(  J\right)  }\phi_{t}^{\left(  J\right)  }%
}{N^{2}}+O_{J-1}\left(  N^{-3}\right)  .
\]
The other parts are easily seen to give $O_{J-1}\left(  N^{-3}\right)  .$ We
we have proved that%
\[
\mathbb{E}_{J-1}\left(  \hat{g}_{ij}^{\left(  J\right)  }\hat{g}_{jt}^{\left(
J\right)  }\right)  =-\sum_{m=1}^{J}\frac{\phi_{i}^{\left(  m\right)  }%
\phi_{t}^{\left(  m\right)  }}{N^{2}}+O_{J-1}\left(  N^{-3}\right)  .
\]

c) We have here $\left\{  i,j\right\}  \cap\left\{  s,t\right\}  =\emptyset.$%
\begin{align*}
\mathbb{E}_{J-1}\left(  \hat{g}_{ij}^{\left(  J\right)  }\hat{g}_{jt}^{\left(
J\right)  }\right)   & =\mathbb{E}_{J-1}\left(  g_{ij}^{\left(  J\right)
}-\frac{\phi_{i}^{\left(  J\right)  }\xi_{j}^{\left(  J\right)  }+\phi
_{j}^{\left(  J\right)  }\xi_{i}^{\left(  J\right)  }}{N}+\frac{\phi
_{i}^{\left(  J\right)  }\phi_{j}^{\left(  J\right)  }}{N^{2}}\sum_{u}\phi
_{u}^{\left(  J\right)  }\xi_{u}^{\left(  J\right)  }\right) \\
& \times\left(  g_{st}^{\left(  J\right)  }-\frac{\phi_{s}^{\left(  J\right)
}\xi_{t}^{\left(  J\right)  }+\phi_{t}^{\left(  J\right)  }\xi_{s}^{\left(
J\right)  }}{N}+\frac{\phi_{s}^{\left(  J\right)  }\phi_{t}^{\left(  J\right)
}}{N^{2}}\sum_{u}\phi_{u}^{\left(  J\right)  }\xi_{u}^{\left(  J\right)
}\right)  .
\end{align*}
The $1\times1,1\times2,2\times1,$ and $2\times2$-terms are clearly of the
desired form, either from induction hypothesis or Lemma \ref{Le_G1}.%
\[
1\times3=\frac{\phi_{s}^{\left(  J\right)  }\phi_{t}^{\left(  J\right)  }%
}{N^{2}}\sum_{u}\mathbb{E}_{J-1}\left(  g_{ij}^{\left(  J\right)  }\phi
_{u}^{\left(  J\right)  }\xi_{u}^{\left(  J\right)  }\right)  .
\]
For $u=i$ we get for the expectation $\phi_{i}^{\left(  J\right)  }\phi
_{j}^{\left(  J\right)  }/N+O_{J-1}\left(  N^{-2}\right)  ,$ so this is of the
desired form. The same applies to $u=j.$ It therefore remains%
\begin{align*}
& \frac{\phi_{s}^{\left(  J\right)  }\phi_{t}^{\left(  J\right)  }}{N^{2}}%
\sum_{u\neq i,j}\phi_{u}^{\left(  J\right)  }\mathbb{E}_{J-1}\left(
g_{ij}^{\left(  J\right)  }\xi_{u}^{\left(  J\right)  }\right) \\
& =\frac{\phi_{s}^{\left(  J\right)  }\phi_{t}^{\left(  J\right)  }}{N^{2}%
}\sum_{u\neq i,j}\phi_{u}^{\left(  J\right)  }\left\{  -\frac{1}{N^{2}}%
\sum_{m=1}^{J-1}\phi_{u}^{\left(  m\right)  }\left[  \phi_{i}^{\left(
J\right)  }\phi_{j}^{\left(  m\right)  }+\phi_{j}^{\left(  J\right)  }\phi
_{i}^{\left(  m\right)  }\right]  \right\}  +O_{J-1}\left(  N^{-4}\right)  .
\end{align*}
As $\sum_{u}\phi_{u}^{\left(  J\right)  }\phi_{u}^{\left(  m\right)  }=0,$ the
whole expression is $O_{J-1}\left(  N^{-4}\right)  .$ The other cases are
handled similarly.
\end{proof}

\section{Proof of Proposition \ref{Prop_Mhat_main}}

We assume $\operatorname{COND}\left(  J\right)  ,$ and (\ref{GaussMain1}) -
(\ref{GaussMain3}) for $k\leq J.$ By Proposition \ref{Th_Gauss_Main} of the
last section, this implies (\ref{GaussMain1}) - (\ref{GaussMain3}) for $k\leq
J+1. $ Using this, we prove now (\ref{Main1}) and (\ref{Main2}) for $k=J+1,$
so that we have proved $\operatorname{COND}\left(  J+1\right)  .$ Having
achieved this, the proof of Proposition \ref{Prop_Mhat_main} is complete.

For $j<k,$ define $\mathbf{X}^{\left(  k,0\right)  }\overset{\mathrm{def}}%
{=}0,\ \mathbf{X}^{\left(  k,j\right)  }\overset{\mathrm{def}}{=}\sum
_{t=1}^{j}\gamma_{t}\mathbf{\phi}^{\left(  t\right)  },$ and $\mathbf{X}%
^{\left(  k,k\right)  }\overset{\mathrm{def}}{=}\sum_{t=1}^{k-1}\gamma
_{t}\mathbf{\phi}^{\left(  t\right)  }+\sqrt{q-\Gamma_{k-1}^{2}}\mathbf{\phi
}^{\left(  k\right)  }.$

Remark that under $\operatorname{COND}\left(  J\right)  $%
\begin{equation}
\mathbf{m}^{\left(  k\right)  }\approx\mathbf{X}^{\left(  k,k\right)
}.\label{m_to_X}%
\end{equation}
for $k\leq J.$ Indeed%
\begin{align*}
& \left\Vert \mathbf{m}^{\left(  k\right)  }-\sum\nolimits_{m=1}%
^{k-1}\left\langle \mathbf{m}^{\left(  k\right)  },\mathbf{\phi}^{\left(
m\right)  }\right\rangle \mathbf{\phi}^{\left(  m\right)  }\right\Vert
\mathbf{\phi}^{\left(  k\right)  }\\
& =\mathbf{m}^{\left(  k\right)  }-\sum\nolimits_{m=1}^{k-1}\left\langle
\mathbf{m}^{\left(  k\right)  },\mathbf{\phi}^{\left(  m\right)
}\right\rangle \mathbf{\phi}^{\left(  m\right)  }\\
& \approx\mathbf{m}^{\left(  k\right)  }-\sum\nolimits_{m=1}^{k-1}\gamma
_{m}\mathbf{\phi}^{\left(  m\right)  }.
\end{align*}
From $q>\Gamma_{k-1}^{2},$ by (\ref{Main1}) and (\ref{Main2}) for $k\leq J,$
and the fact that the $\phi_{j}^{\left(  k\right)  }$ are uniformly bounded on
$A_{J},$ we have%
\[
\left\Vert \mathbf{m}^{\left(  k\right)  }-\sum\nolimits_{m=1}^{k-1}%
\left\langle \mathbf{m}^{\left(  k\right)  },\mathbf{\phi}^{\left(  m\right)
}\right\rangle \mathbf{\phi}^{\left(  m\right)  }\right\Vert \simeq
\sqrt{q-\Gamma_{k-1}^{2}},
\]
So the claim (\ref{m_to_X}) follows.

We define for $1\leq s<k$%
\[
\mathbf{m}^{\left(  k,s\right)  }\overset{\mathrm{def}}{=}\operatorname*{Th}%
\left(  \mathbf{g}^{\left(  s\right)  }\mathbf{M}^{\left(  k-1,s-1\right)
}+\sum\nolimits_{t=1}^{s-1}\gamma_{t}\mathbf{\xi}^{\left(  t\right)  }%
+\beta\left(  1-q\right)  \left\{  \mathbf{X}^{\left(  k-2,s-1\right)
}-\mathbf{m}^{\left(  k-2\right)  }\right\}  \right)  .
\]
Remark that by Lemma \ref{Le_Orth1}, we have $\mathbf{g}^{\left(  s\right)
}\mathbf{M}^{\left(  k-1,s-1\right)  }=\mathbf{g}^{\left(  s\right)
}\mathbf{m}^{\left(  k\right)  }.$ Evidently%
\[
\mathbf{m}^{\left(  k,1\right)  }=\mathbf{m}^{\left(  k\right)  },
\]
and we define%
\[
\mathbf{\hat{m}}^{\left(  1\right)  }\overset{\mathrm{def}}{=}\mathbf{m}%
^{\left(  1\right)  }=\sqrt{q}\mathbf{1,\ \hat{m}}^{\left(  k\right)
}\overset{\mathrm{def}}{=}\mathbf{m}^{\left(  k,k-1\right)  },\ k\geq2.
\]
By (\ref{m_to_X})%
\begin{align*}
\mathbf{\hat{m}}^{\left(  k\right)  }  & =\operatorname*{Th}\left(
\mathbf{g}^{\left(  k-1\right)  }\mathbf{m}^{\left(  k-1\right)  }%
+\sum\nolimits_{t=1}^{k-2}\gamma_{t}\mathbf{\xi}^{\left(  t\right)  }\right)
\\
& =\operatorname*{Th}\left(  \mathbf{g}^{\left(  k-1\right)  }\mathbf{M}%
^{\left(  k-1\right)  }+\sum\nolimits_{t=1}^{k-2}\gamma_{t}\mathbf{\xi
}^{\left(  t\right)  }\right) \\
& =\operatorname*{Th}\left(  \left\Vert \mathbf{M}^{\left(  k-1\right)
}\right\Vert \mathbf{\xi}^{\left(  k-1\right)  }+\sum\nolimits_{t=1}%
^{k-2}\gamma_{t}\mathbf{\xi}^{\left(  t\right)  }\right)  .
\end{align*}

The key result of our paper is

\begin{proposition}
\label{Prop_key}%
\begin{equation}
\mathbf{m}^{\left(  k\right)  }\approx\mathbf{\hat{m}}^{\left(  k\right)
}\label{key}%
\end{equation}
holds for all $k.$
\end{proposition}

This proposition is correct for all $\beta.$ The key point with (\ref{AT}) is
that the first summand $\left\Vert \mathbf{M}^{\left(  k-1\right)
}\right\Vert \mathbf{\xi}^{\left(  k-1\right)  }$ disappears for
$k\rightarrow\infty$ as $\left\Vert \mathbf{M}^{\left(  k-1\right)
}\right\Vert \simeq\sqrt{q-\Gamma_{k-2}^{2}},$ so that for large $k,$
$\mathbf{\hat{m}}^{\left(  k\right)  }$ stabilizes to $\operatorname*{Th}%
\left(  \sum\nolimits_{t=1}^{k-2}\gamma_{t}\mathbf{\xi}^{\left(  t\right)
}\right)  , $ but above the AT-line $q-\Gamma_{k-2}^{2}$ does not converge to
$0.$ Therefore, above the AT-line, in every iteration, new conditionally
independent contributions appear.

The above proposition is proved by showing that $\operatorname{COND}\left(
J\right)  $ implies%
\begin{equation}
\mathbf{m}^{\left(  J+1\right)  }\approx\mathbf{\hat{m}}^{\left(  J+1\right)
}.\label{First_Ind}%
\end{equation}
As $\operatorname{COND}\left(  J\right)  $ implies trivially
$\operatorname{COND}\left(  J^{\prime}\right)  $ for $J^{\prime}<J,$ it is
then clear that $\operatorname{COND}\left(  J\right)  $ implies $\mathbf{m}%
^{\left(  k\right)  }\approx\mathbf{\hat{m}}^{\left(  k\right)  }$ for all
$k\leq J+1.$ As the $m_{j}^{\left(  k\right)  } $ are uniformly bounded by
$1,$ we get from that%
\[
\left\langle \mathbf{m}^{\left(  J+1\right)  },\mathbf{m}^{\left(  j\right)
}\right\rangle \simeq\left\langle \mathbf{\hat{m}}^{\left(  J+1\right)
},\mathbf{\hat{m}}^{\left(  j\right)  }\right\rangle ,
\]
for all $j\leq J+1.$ We will then prove (Lemma \ref{Le_Main}) that%
\[
\left\langle \mathbf{\hat{m}}^{\left(  J+1\right)  },\mathbf{\hat{m}}^{\left(
j\right)  }\right\rangle \simeq\rho_{j}%
\]
for $j\leq J,$ and%
\[
\left\Vert \mathbf{\hat{m}}^{\left(  J+1\right)  }\right\Vert ^{2}\simeq q.
\]
This will prove $\operatorname{COND}\left(  J+1\right)  ,$ and therefore, this
will have finished the whole induction procedure.

Together with proving (\ref{First_Ind}), we also show%
\begin{equation}
\left\langle \mathbf{\xi}^{\left(  m\right)  },\mathbf{m}^{\left(  k\right)
}\right\rangle \simeq\left\langle \mathbf{\xi}^{\left(  m\right)
},\mathbf{\hat{m}}^{\left(  k\right)  }\right\rangle ,\ \forall
m<k\label{Second_Ind}%
\end{equation}
for $k=J+1$ which is not evident from (\ref{First_Ind}) as the $\xi
_{i}^{\left(  m\right)  }$ are not bounded.

\begin{lemma}
Assume the validity of (\ref{Main1})-(\ref{Main3}) and (\ref{Second_Ind}) for
$k\leq J.$ Then for $s=1,\ldots,J-1$%
\[
\mathbf{m}^{\left(  J+1,s\right)  }\approx\mathbf{m}^{\left(  J+1,s+1\right)
}.
\]
In particular, it follows%
\[
\mathbf{m}^{\left(  J+1\right)  }\approx\mathbf{\hat{m}}^{\left(  J+1\right)
}.
\]
Furthermore (\ref{Second_Ind}) holds for $k=J+1.$
\end{lemma}

\begin{proof}
We prove by induction on $s,\ 1\leq s\leq J-1,$ that%
\begin{equation}
\mathbf{m}^{\left(  J+1,s\right)  }\approx\mathbf{m}^{\left(  J+1,s+1\right)
},\label{equ1}%
\end{equation}
and%
\begin{equation}
\left\langle \mathbf{\xi}^{\left(  m\right)  },\mathbf{m}^{\left(
J+1,s\right)  }\right\rangle \simeq\left\langle \mathbf{\xi}^{\left(
m\right)  },\mathbf{m}^{\left(  J+1,s+1\right)  }\right\rangle ,\ m\leq
J.\label{equ2}%
\end{equation}

We have%
\[
\mathbf{g}^{\left(  s+1\right)  }=\mathbf{g}^{\left(  s\right)  }-\mathbf{\xi
}^{\left(  s\right)  }\otimes_{\mathrm{s}}\mathbf{\phi}^{\left(  s\right)
}+\left\langle \mathbf{\xi}^{\left(  s\right)  },\mathbf{\phi}^{\left(
s\right)  }\right\rangle \left(  \mathbf{\phi}^{\left(  s\right)  }%
\otimes\mathbf{\phi}^{\left(  s\right)  }\right)  +\mathbf{c}^{\left(
s\right)  }%
\]
where%
\[
c_{ij}^{\left(  s\right)  }=\sum_{r}x_{ij,r}^{\left(  s\right)  }\xi
_{r}^{\left(  s\right)  },
\]
see Lemma \ref{Le_G2}. Therefore%
\[
\mathbf{m}^{\left(  J+1,s\right)  }=\operatorname*{Th}\left(  \mathbf{g}%
^{\left(  s+1\right)  }\mathbf{m}^{\left(  J\right)  }+\mathbf{y+}\beta\left(
1-q\right)  \left\{  \mathbf{X}^{\left(  J-1,s-1\right)  }-\mathbf{m}^{\left(
J-1\right)  }\right\}  \right)
\]
where%
\[
\mathbf{y}\overset{\mathrm{def}}{=}\left\langle \mathbf{\phi}^{\left(
s\right)  },\mathbf{m}^{\left(  J\right)  }\right\rangle \mathbf{\xi}^{\left(
s\right)  }+\left\langle \mathbf{\xi}^{\left(  s\right)  },\mathbf{m}^{\left(
J\right)  }\right\rangle \mathbf{\phi}^{\left(  s\right)  }+\left\langle
\mathbf{\phi}^{\left(  s\right)  },\mathbf{m}^{\left(  J\right)
}\right\rangle \left\langle \mathbf{\phi}^{\left(  s\right)  },\mathbf{\xi
}^{\left(  J\right)  }\right\rangle \mathbf{\phi}^{\left(  s\right)
}+\mathbf{c}^{\left(  s\right)  }\mathbf{m}^{\left(  J\right)  }.
\]
We write%
\[
\mathbf{y}^{\left(  1\right)  }\overset{\mathrm{def}}{=}\left\langle
\mathbf{\phi}^{\left(  s\right)  },\mathbf{m}^{\left(  J\right)
}\right\rangle \mathbf{\xi}^{\left(  s\right)  }+\left\langle \mathbf{\xi
}^{\left(  s\right)  },\mathbf{m}^{\left(  J\right)  }\right\rangle
\mathbf{\phi}^{\left(  s\right)  }+\left\langle \mathbf{\phi}^{\left(
s\right)  },\mathbf{m}^{\left(  J\right)  }\right\rangle \left\langle
\mathbf{\phi}^{\left(  s\right)  },\mathbf{\xi}^{\left(  J\right)
}\right\rangle \mathbf{\phi}^{\left(  s\right)  },
\]%
\[
\mathbf{y}^{\left(  2\right)  }\overset{\mathrm{def}}{=}\left\langle
\mathbf{\phi}^{\left(  s\right)  },\mathbf{m}^{\left(  J\right)
}\right\rangle \mathbf{\xi}^{\left(  s\right)  }+\left\langle \mathbf{\xi
}^{\left(  s\right)  },\mathbf{m}^{\left(  J\right)  }\right\rangle
\mathbf{\phi}^{\left(  s\right)  },
\]%
\[
\mathbf{y}^{\left(  3\right)  }\overset{\mathrm{def}}{=}\gamma_{s}\mathbf{\xi
}^{\left(  s\right)  }+\left\langle \mathbf{\xi}^{\left(  s\right)
},\mathbf{m}^{\left(  J\right)  }\right\rangle \mathbf{\phi}^{\left(
s\right)  },
\]%
\[
\mathbf{y}^{\left(  4\right)  }\overset{\mathrm{def}}{=}\gamma_{s}\mathbf{\xi
}^{\left(  s\right)  }+\left\langle \mathbf{\xi}^{\left(  s\right)
},\mathbf{\hat{m}}^{\left(  J\right)  }\right\rangle \mathbf{\phi}^{\left(
s\right)  },
\]%
\[
\mathbf{y}^{\left(  5\right)  }\overset{\mathrm{def}}{=}\gamma_{s}\mathbf{\xi
}^{\left(  s\right)  }+\beta\left(  1-q\right)  \gamma_{s}\mathbf{\phi
}^{\left(  s\right)  },
\]
and then set ad hoc%
\[
\mathbf{\mu}^{\left(  0\right)  }\overset{\mathrm{def}}{=}\mathbf{m}^{\left(
J+1,s\right)  },
\]
and define $\mathbf{\mu}^{\left(  n\right)  }$ where $\mathbf{y}$ is replaced
by $\mathbf{y}^{\left(  n\right)  },$ $n=1,\ldots,5.$ Remark that%
\[
\mathbf{\mu}^{\left(  5\right)  }=\mathbf{m}^{\left(  J+1,s+1\right)  }.
\]
We will prove%
\begin{equation}
\mathbf{\mu}^{\left(  n-1\right)  }\approx\mathbf{\mu}^{\left(  n\right)
},\ n=1,\ldots,5,\label{equiv1}%
\end{equation}
and%
\begin{equation}
\left\langle \mathbf{\xi}^{\left(  m\right)  },\mathbf{\mu}^{\left(
n-1\right)  }\right\rangle \simeq\left\langle \mathbf{\xi}^{\left(  m\right)
},\mathbf{\mu}^{\left(  n\right)  }\right\rangle ,\ n=1,\ldots
,5.\label{equiv2}%
\end{equation}
which prove the desired induction in $s.$

To switch from $\mathbf{\mu}^{\left(  0\right)  }$ to $\mathbf{\mu}^{\left(
1\right)  },$ we observe that by the estimates of Lemma \ref{Le_G2}, one has%
\[
\left\vert \left(  \mathbf{c}^{\left(  s\right)  }\mathbf{m}^{\left(
J\right)  }\right)  _{i}\right\vert \leq O_{s-1}\left(  1\right)  \left[
\frac{1}{N}\left\vert \xi_{i}^{\left(  s\right)  }\right\vert +\frac{1}{N^{2}%
}\sum\nolimits_{j}\left\vert \xi_{j}^{\left(  s\right)  }\right\vert \right]
.
\]
Therefore%
\[
\frac{1}{N}\sum_{i}\left\vert \mu_{i}^{\left(  0\right)  }-\mu_{i}^{\left(
1\right)  }\right\vert \leq\frac{O_{s-1}\left(  1\right)  }{N^{2}}%
\sum\nolimits_{j}\left\vert \xi_{j}^{\left(  s\right)  }\right\vert ,
\]
and%
\[
\frac{1}{N}\sum_{i}\left\vert \xi_{i}^{\left(  m\right)  }\left(  \mu
_{i}^{\left(  0\right)  }-\mu_{i}^{\left(  1\right)  }\right)  \right\vert
\leq\frac{O_{s-1}\left(  1\right)  }{N}\left\{  \frac{1}{N}\sum_{i}\left\vert
\xi_{i}^{\left(  m\right)  }\xi_{i}^{\left(  s\right)  }\right\vert +\frac
{1}{N}\sum_{i}\left\vert \xi_{i}^{\left(  m\right)  }\right\vert \frac{1}%
{N}\sum_{i}\left\vert \xi_{i}^{\left(  s\right)  }\right\vert \right\}  .
\]
By choosing $K$ large enough, we get for $1/\sqrt{N}\leq t\leq1$ by Corollary
\ref{Le_LD} a)%
\begin{align*}
\mathbb{P}\left(  \frac{1}{N}\sum_{i}\left\vert \mu_{i}^{\left(  0\right)
}-\mu_{i}^{\left(  1\right)  }\right\vert \geq t\right)   & \leq
\mathbb{P}\left(  \frac{K}{N}\sum\nolimits_{j}\left\vert \xi_{j}^{\left(
s\right)  }\right\vert \geq tN\right)  +\mathbb{P}\left(  O_{s-1}\left(
1\right)  \geq K\right) \\
& \leq C\exp\left[  -N/C\right]  \leq C\exp\left[  -Nt^{2}/C\right]  .
\end{align*}
For $t\leq1/\sqrt{N},$ the bound is trivial anyway. This proves (\ref{equiv1})
for $n=1.$ (\ref{equiv2}) follows in the same way using Corollary \ref{Le_LD}
b).%
\begin{align*}
\frac{1}{N}\sum_{i}\left\vert \mu_{i}^{\left(  1\right)  }-\mu_{i}^{\left(
2\right)  }\right\vert  & \leq C\left\vert \left\langle \mathbf{\phi}^{\left(
s\right)  },\mathbf{m}^{\left(  J\right)  }\right\rangle \left\langle
\mathbf{\phi}^{\left(  s\right)  },\mathbf{\xi}^{\left(  J\right)
}\right\rangle \left\langle \mathbf{\phi}^{\left(  s\right)  },\mathbf{1}%
\right\rangle \right\vert \\
& \leq C\left\vert \left\langle \mathbf{\phi}^{\left(  s\right)  }%
,\mathbf{\xi}^{\left(  J\right)  }\right\rangle \right\vert
\end{align*}
on $A_{J}.$ (\ref{equiv1}) for $n=2$ then follows from Corollary \ref{Le_LD}
c). As for (\ref{equiv2}), we remark that%
\[
\frac{1}{N}\sum_{i}\left\vert \xi_{i}^{\left(  m\right)  }\left(  \mu
_{i}^{\left(  1\right)  }-\mu_{i}^{\left(  2\right)  }\right)  \right\vert
\leq C\left\vert \left\langle \mathbf{\phi}^{\left(  s\right)  },\mathbf{\xi
}^{\left(  J\right)  }\right\rangle \right\vert \left\vert \left\langle
\mathbf{\phi}^{\left(  s\right)  },\mathbf{\xi}^{\left(  m\right)
}\right\rangle \right\vert .
\]
We can then again use Corollary \ref{Le_LD} c) remarking that $\exp\left[
-Nt/C\right]  \leq\exp\left[  -Nt^{2}/C\right]  $ for $t\leq1.$%
\[
\frac{1}{N}\sum_{i}\left\vert \mu_{i}^{\left(  2\right)  }-\mu_{i}^{\left(
3\right)  }\right\vert \leq C\left\vert \left\langle \mathbf{\phi}^{\left(
s\right)  },\mathbf{m}^{\left(  J\right)  }\right\rangle -\gamma
_{s}\right\vert \frac{1}{N}\sum_{i}\left\vert \xi_{i}^{\left(  s\right)
}\right\vert .
\]
(\ref{equiv1}) for $n=3$ follows from the induction hypothesis (\ref{Main3}),
and Corollary \ref{Le_LD} a). Similarly with (\ref{equiv2}) but here, one has
to use part b) of Corollary \ref{Le_LD}.%
\[
\frac{1}{N}\sum_{i}\left\vert \mu_{i}^{\left(  3\right)  }-\mu_{i}^{\left(
4\right)  }\right\vert \leq C\left\vert \left\langle \mathbf{\xi}^{\left(
s\right)  },\mathbf{m}^{\left(  J\right)  }-\mathbf{\hat{m}}^{\left(
J\right)  }\right\rangle \right\vert
\]
on $A_{k},$ and one uses the induction hypothesis (\ref{Second_Ind}) for $J$
to get (\ref{equiv1}) for $n=4.$ Remark that actually, one has a bound uniform
in $i:$%
\[
\left\vert \mu_{i}^{\left(  3\right)  }-\mu_{i}^{\left(  4\right)
}\right\vert \leq C\left\vert \left\langle \mathbf{\xi}^{\left(  s\right)
},\mathbf{m}^{\left(  J\right)  }-\mathbf{\hat{m}}^{\left(  J\right)
}\right\rangle \right\vert .
\]
Therefore, one also gets (\ref{equiv2}) using Corollary \ref{Le_LD}. Up to
now, we have obtained%
\begin{align*}
\mathbf{m}^{\left(  J+1,s\right)  }  & \approx\operatorname*{Th}%
\Big(\mathbf{g}^{\left(  s\right)  }\mathbf{M}^{\left(  k-1,s-1\right)  }%
+\sum\nolimits_{t=1}^{s}\gamma_{t}\mathbf{\xi}^{\left(  t\right)  }\\
& +\left\langle \mathbf{\xi}^{\left(  s\right)  },\mathbf{\hat{m}}^{\left(
J\right)  }\right\rangle \mathbf{\phi}^{\left(  s\right)  }+\beta\left(
1-q\right)  \left\{  \mathbf{X}^{\left(  J-1,s-1\right)  }-\mathbf{m}^{\left(
J-1\right)  }\right\}  \Big)
\end{align*}
and%
\begin{align*}
\left\langle \mathbf{\xi}^{\left(  m\right)  },\mathbf{m}^{\left(
J+1,s\right)  }\right\rangle  & \simeq\left\langle \mathbf{\xi}^{\left(
m\right)  },\operatorname*{Th}\Big(\mathbf{g}^{\left(  s\right)  }%
\mathbf{M}^{\left(  k-1,s-1\right)  }+\sum\nolimits_{t=1}^{s}\gamma
_{t}\mathbf{\xi}^{\left(  t\right)  }\right. \\
& \left.  +\left\langle \mathbf{\xi}^{\left(  s\right)  },\mathbf{\hat{m}%
}^{\left(  J\right)  }\right\rangle \mathbf{\phi}^{\left(  s\right)  }%
+\beta\left(  1-q\right)  \left\{  \mathbf{X}^{\left(  J-1,s-1\right)
}-\mathbf{m}^{\left(  J-1\right)  }\right\}  \Big)\right\rangle
\end{align*}
By Lemma \ref{Le_Main} a) below, we have%
\begin{equation}
\left\langle \mathbf{\xi}^{\left(  s\right)  },\mathbf{\hat{m}}^{\left(
J\right)  }\right\rangle \simeq\left\{
\begin{array}
[c]{cc}%
\beta\left(  1-q\right)  \gamma_{s} & \mathrm{for\ }s<J-1\\
\beta\left(  1-q\right)  \sqrt{q-\Gamma_{J-2}^{2}} & \mathrm{for\ }s=J-1
\end{array}
\right.  ,\label{Aprox3}%
\end{equation}
and we can therefore replace $\left\langle \mathbf{\xi}^{\left(  s\right)
},\mathbf{\hat{m}}^{\left(  J\right)  }\right\rangle \mathbf{\phi}^{\left(
s\right)  }$ on the right hand side, by $\beta\left(  1-q\right)  \gamma
_{s}\mathbf{\phi}^{\left(  s\right)  }$ for $s<J-1,$ or $\beta\left(
1-q\right)  \sqrt{q-\Gamma_{J-2}^{2}}\mathbf{\phi}^{\left(  J-1\right)  }$ for
$s=J-1,$ which is the same as replacing $\mathbf{X}^{\left(  J-1,s-1\right)
}$ by $\mathbf{X}^{\left(  J-1,s\right)  }.$ Therefore, the lemma is proved.
\end{proof}

\begin{lemma}
\label{Le_Main} We assume $\operatorname{COND}\left(  J\right)  $.

\begin{enumerate}
\item[a)]
\[
\left\langle \mathbf{\xi}^{\left(  s\right)  },\mathbf{\hat{m}}^{\left(
J\right)  }\right\rangle \simeq\left\{
\begin{array}
[c]{cc}%
\beta\left(  1-q\right)  \gamma_{s} & \mathrm{for\ }s<J-1\\
\beta\left(  1-q\right)  \sqrt{q-\Gamma_{J-2}^{2}} & \mathrm{for\ }s=J-1
\end{array}
\right.  .
\]

\item[b)]
\[
\left\langle \mathbf{\hat{m}}^{\left(  J+1\right)  },\mathbf{\hat{m}}^{\left(
j\right)  }\right\rangle \simeq\rho_{j}%
\]
for $j\leq J,$ and%
\[
\left\langle \mathbf{\hat{m}}^{\left(  J+1\right)  },\mathbf{\hat{m}}^{\left(
J+1\right)  }\right\rangle \simeq q.
\]

\end{enumerate}
\end{lemma}

\begin{proof}
a) Consider first the case $s=J-1.$%
\[
\mathbf{\hat{m}}^{\left(  J\right)  }=\operatorname*{Th}\left(  \left\Vert
\mathbf{M}^{\left(  J-1\right)  }\right\Vert \mathbf{\xi}^{\left(  J-1\right)
}+\sum\nolimits_{t=1}^{J-2}\gamma_{t}\mathbf{\xi}^{\left(  t\right)  }\right)
.
\]%
\[
\frac{1}{N}\sum_{i=1}^{N}\xi_{i}^{\left(  J-1\right)  }\hat{m}_{i}^{\left(
J\right)  }=\frac{1}{N}\sum_{i=1}^{N}\xi_{i}^{\left(  J-1\right)
}\operatorname*{Th}\left(  \left\Vert \mathbf{M}^{\left(  J-1\right)
}\right\Vert \xi_{i}^{\left(  J-1\right)  }+\sum\nolimits_{t=1}^{J-2}%
\gamma_{t}\xi_{i}^{\left(  t\right)  }\right)  .
\]
We condition on $\mathcal{F}_{J-2}.$ Then $\mathbf{\xi}^{\left(  J-1\right)  }
$ is conditionally Gaussian with covariances given in Lemma \ref{Le_G1} a),
b). We can therefore apply Lemma \ref{Le_MainLLN} which gives, conditionally
on $\mathcal{F}_{J-2},$ on an event $B_{J-2}\in\mathcal{F}_{J-2}$ which has
probability $\geq1-C\exp\left[  -N/C\right]  ,$%
\begin{align*}
\frac{1}{N}\sum_{i=1}^{N}\xi_{i}^{\left(  J-1\right)  }\hat{m}_{i}^{\left(
J\right)  }  & \simeq\frac{1}{N}\sum_{i=1}^{N}EZ_{J-1}\operatorname*{Th}%
\left(  \left\Vert \mathbf{M}^{\left(  J-1\right)  }\right\Vert Z_{J-1}%
+\sum\nolimits_{t=1}^{J-2}\gamma_{t}\xi_{i}^{\left(  t\right)  }\right) \\
& =\frac{1}{N}\sum_{i=1}^{N}\beta\left\Vert \mathbf{M}^{\left(  J-1\right)
}\right\Vert \left[  1-E\operatorname*{Th}\nolimits^{2}\left(  \left\Vert
\mathbf{M}^{\left(  J-1\right)  }\right\Vert Z_{J-1}+\sum\nolimits_{t=1}%
^{J-2}\gamma_{t}\xi_{i}^{\left(  t\right)  }\right)  \right] \\
& \simeq\beta\sqrt{q-\Gamma_{J-1}^{2}}\frac{1}{N}\sum_{i=1}^{N}\left[
1-E\operatorname*{Th}\nolimits^{2}\left(  \sqrt{q-\Gamma_{J-1}^{2}}%
Z_{J-1}+\sum\nolimits_{t=1}^{J-2}\gamma_{t}\xi_{i}^{\left(  t\right)
}\right)  \right]  .
\end{align*}

Applying now Lemma \ref{Le_MainLLN} successively to $\mathbf{\xi}^{\left(
J-2\right)  },\mathbf{\xi}^{\left(  J-2\right)  },\ldots$ , we get%
\begin{align*}
\frac{1}{N}\sum_{i=1}^{N}\xi_{i}^{\left(  J-1\right)  }\hat{m}_{i}^{\left(
J\right)  }  & \simeq\beta\sqrt{q-\Gamma_{J-1}^{2}}\left[
1-E\operatorname*{Th}\nolimits^{2}\left(  \sqrt{q-\Gamma_{J-1}^{2}}%
Z_{J-1}+\sum\nolimits_{t=1}^{J-2}\gamma_{t}Z_{t}\right)  \right] \\
& =\beta\sqrt{q-\Gamma_{J-1}^{2}}\left(  1-q\right)  .
\end{align*}

The case $s<J-1$ uses a minor modification of the argument. One first uses
Lemma \ref{Le_MainLLN} successively to get%
\begin{align*}
\frac{1}{N}\sum_{i=1}^{N}\xi_{i}^{\left(  s\right)  }\hat{m}_{i}^{\left(
J\right)  }  & \simeq\frac{1}{N}\sum_{i=1}^{N}\xi_{i}^{\left(  s\right)
}E\operatorname*{Th}\left(  \left\Vert \mathbf{M}^{\left(  J-1\right)
}\right\Vert Z_{J-1}+\sum\nolimits_{t=s+1}^{J-2}\gamma_{t}Z_{t}+\gamma_{s}%
\xi_{i}^{\left(  s\right)  }+\sum\nolimits_{t=1}^{s-1}\gamma_{t}\xi
_{i}^{\left(  t\right)  }\right) \\
& \simeq EZ_{s}\operatorname*{Th}\left(  \left\Vert \mathbf{M}^{\left(
J-1\right)  }\right\Vert Z_{J-1}+\sum\nolimits_{t=1}^{J-2}\gamma_{t}%
Z_{t}\right) \\
& =\beta\gamma_{s}\left[  1-E\operatorname*{Th}\nolimits^{2}\left(  \left\Vert
\mathbf{M}^{\left(  J-1\right)  }\right\Vert Z_{J-1}+\sum\nolimits_{t=1}%
^{J-2}\gamma_{t}Z_{t}\right)  \right]  =\beta\gamma_{s}\left(  1-q\right)  .
\end{align*}

b) This also comes with a modification of the reasoning in a).

Assume first $j\leq J.$%
\begin{align*}
\frac{1}{N}\sum_{i=1}^{N}\hat{m}_{i}^{\left(  J+1\right)  }\hat{m}%
_{i}^{\left(  j\right)  }  & \simeq\frac{1}{N}\sum_{i=1}^{N}%
\Big[E\operatorname*{Th}\left(  \left\Vert \mathbf{M}^{\left(  J\right)
}\right\Vert Z_{J}+\sum\nolimits_{t=1}^{J-1}\gamma_{t}\xi_{i}^{\left(
t\right)  }\right) \\
& \times\operatorname*{Th}\left(  \left\Vert \mathbf{M}^{\left(  j-1\right)
}\right\Vert \xi_{i}^{\left(  j-1\right)  }+\sum\nolimits_{t=1}^{j-2}%
\gamma_{t}\xi_{i}^{\left(  t\right)  }\right)  \Big].
\end{align*}
In the case $j=J+1,$ the outcome is similar, one only has to replace the
second factor by $\operatorname*{Th}\left(  \left\Vert \mathbf{M}^{\left(
J\right)  }\right\Vert Z_{J}+\sum\nolimits_{t=1}^{J-1}\gamma_{t}\mathbf{\xi
}_{i}^{\left(  t\right)  }\right)  .$

The next observation is that by the induction hypothesis, one can replace
$\left\Vert \mathbf{M}^{\left(  J\right)  }\right\Vert $ by $\sqrt
{q-\Gamma_{J-1}^{2}}$ and we get%
\begin{align*}
\frac{1}{N}\sum_{i=1}^{N}\hat{m}_{i}^{\left(  J+1\right)  }\hat{m}%
_{i}^{\left(  j\right)  }  & \simeq\frac{1}{N}\sum_{i=1}^{N}%
\Big[E\operatorname*{Th}\left(  \sqrt{q-\Gamma_{J-1}^{2}}Z_{J}+\sum
\nolimits_{t=1}^{J-1}\gamma_{t}\xi_{i}^{\left(  t\right)  }\right) \\
& \times\operatorname*{Th}\left(  \left\Vert \mathbf{M}^{\left(  j-1\right)
}\right\Vert \xi_{i}^{\left(  j-1\right)  }+\sum\nolimits_{t=1}^{j-2}%
\gamma_{t}\mathbf{\xi}_{i}^{\left(  t\right)  }\right)  \Big]
\end{align*}
in the $j\leq J$ case, and%
\[
\frac{1}{N}\sum_{i=1}^{N}\hat{m}_{i}^{\left(  J+1\right)  2}\simeq\frac{1}%
{N}\sum_{i=1}^{N}E\operatorname*{Th}\nolimits^{2}\left(  \sqrt{q-\Gamma
_{J-1}^{2}}Z_{J}+\sum\nolimits_{t=1}^{J-1}\gamma_{t}\xi_{i}^{\left(  t\right)
}\right)  .
\]

The important point is that the factor before $Z_{J}$ is replaced by a
constant, which is due to the induction hypothesis. We can now proceed in the
same way with $\mathbf{\xi}^{\left(  J-1\right)  },$ applying again Lemma
\ref{Le_MainLLN}, conditioned on $\mathcal{F}_{J-2},$ and the induction
hypothesis. The final outcome is%
\begin{align*}
\frac{1}{N}\sum_{i=1}^{N}\hat{m}_{i}^{\left(  J+1\right)  }\hat{m}%
_{i}^{\left(  j\right)  }  & \simeq E\Big[\operatorname*{Th}\left(
\sqrt{q-\Gamma_{J-1}^{2}}Z_{J}+\sum\nolimits_{r=j}^{J-1}\gamma_{r}Z_{r}%
+\sum\nolimits_{r=1}^{j-1}\gamma_{r}Z_{r}\right) \\
& \times\operatorname*{Th}\left(  \sqrt{q-\Gamma_{j-1}^{2}}Z_{j}%
+\sum\nolimits_{r=1}^{j-2}\gamma_{r}Z_{r}\right)  \Big],
\end{align*}
in the case $j\leq J,$ and%
\[
\frac{1}{N}\sum_{i=1}^{N}\hat{m}_{i}^{\left(  J+1\right)  2}\simeq
E\operatorname*{Th}\nolimits^{2}\left(  \sqrt{q-\Gamma_{J-1}^{2}}Z_{J}%
+\sum\nolimits_{r=j}^{J-1}\gamma_{r}Z_{r}+\sum\nolimits_{r=1}^{j-1}\gamma
_{r}Z_{r}\right)  .
\]
For the latter case, the right hand side is simply $q.$ For the case $j\leq
J,$ we can rewrite the expression on the right hand side as%
\begin{equation}
E\operatorname*{Th}\left(  \sqrt{q-\Gamma_{j-1}^{2}}Z^{\prime\prime}%
+\gamma_{j-1}Z^{\prime}+\Gamma_{j-2}Z\right)  \operatorname*{Th}\left(
\sqrt{q-\Gamma_{j-2}^{2}}Z^{\prime}+\Gamma_{j-2}Z\right)  .\label{repr1}%
\end{equation}
We represent%
\begin{align*}
\sqrt{q-\Gamma_{j-1}^{2}}Z^{\prime\prime}+\gamma_{j-1}Z^{\prime}  &
=aZ_{1}+bZ_{2}\\
\sqrt{q-\Gamma_{j-2}^{2}}Z^{\prime}  & =aZ_{1}+bZ_{3}.
\end{align*}
Solving, we get $a^{2}+b^{2}=q-\Gamma_{j-2}^{2},$ and%
\[
a^{2}=\gamma_{j-1}\sqrt{q-\Gamma_{j-2}^{2}}.
\]
Using this, we get that (\ref{repr1}) equals%
\[
E\operatorname*{Th}\left(  \Gamma_{j-2}Z+aZ_{1}+bZ_{2}\right)
\operatorname*{Th}\left(  \Gamma_{j-2}+aZ_{1}+bZ_{2}\right)  =\psi\left(
\Gamma_{j-2}^{2}+a^{2}\right)  .
\]%
\[
\Gamma_{j-2}^{2}+a^{2}=\Gamma_{j-2}^{2}+\gamma_{j-1}\sqrt{q-\Gamma_{j-2}^{2}%
}=\rho_{j-1}.
\]
Therefore, for $j\leq J,$ we get%
\[
\frac{1}{N}\sum_{i=1}^{N}\hat{m}_{i}^{\left(  J+1\right)  }\hat{m}%
_{i}^{\left(  j\right)  }\simeq\psi\left(  \rho_{j-1}\right)  =\rho_{j}.
\]

\end{proof}

\appendix

\section{Appendix}

\begin{lemma}
\label{Le_Gauss_App1}Let $\mathbf{\zeta}=\left(  \zeta_{i}\right)
_{i=1,\ldots,N}$ be Gaussian vectors with $\sup_{N,i}\mathbb{E}\left(
\zeta_{i}^{2}\right)  <\infty,$ and $\sup_{N,i\neq j}N\left\vert
\mathbb{E}\left(  \zeta_{i}\zeta_{j}\right)  \right\vert <\infty.$ Then there
exist $K,C>0$ such that%
\begin{equation}
\mathbb{P}\left(  \frac{1}{N}\sum\nolimits_{i=1}^{N}\left\vert \zeta
_{i}\right\vert \geq K\right)  \leq C\exp\left[  -N/C\right] \label{LD1}%
\end{equation}
and%
\begin{equation}
\mathbb{P}\left(  \frac{1}{N}\sum\nolimits_{i=1}^{N}\zeta_{i}^{2}\geq
K\right)  \leq C\exp\left[  -N/C\right]  .\label{LD2}%
\end{equation}

\end{lemma}

\begin{proof}
We can multiply the $\zeta_{i}$ by a fixed positive real number. Therefore, we
may assume that $\sup_{N,i\neq j}N\left\vert \mathbb{E}\left(  \zeta_{i}%
\zeta_{j}\right)  \right\vert \leq1/4,\ \sup_{N,i}\mathbb{E}\left(  \zeta
_{i}^{2}\right)  \leq1.$ Put $\alpha_{i}\overset{\mathrm{def}}{=}%
1-\mathbb{E}\left(  \zeta_{i}^{2}\right)  ,$ and choose independent Gaussians
$U_{i}$ with $\mathbb{E}U_{i}^{2}=\alpha_{i}.$ If we prove the statements
(\ref{LD1}) and (\ref{LD2}) for the sequence $\left\{  \zeta_{i}%
+U_{i}\right\}  ,$ then it follows for the $\zeta_{i}$ itself, simply because
(\ref{LD1}) and (\ref{LD2}) hold for the $U_{i}.$ Therefore we may assume that
$\mathbb{E}\left(  \zeta_{i}^{2}\right)  =1,$ and $\left\vert \mathbb{E}%
\left(  \zeta_{i}\zeta_{j}\right)  \right\vert \leq1/4N$ for $i\neq j.$ Write
$\Sigma$ for the covariance matrix of $\left\{  \zeta_{i}\right\}  .$
$\Sigma=I+\mathbf{\varepsilon},$ where $\left\vert \varepsilon_{ij}\right\vert
\leq1/4N.$ Taking the symmetric square root%
\[
I+\mathbf{\alpha}=\sqrt{I+\mathbf{\varepsilon}},
\]
then $\sup_{i,j\leq N}\left\vert \alpha_{ij}\right\vert \leq C/N.$ Therefore,
we can represent the $\zeta_{i}$ as%
\[
\zeta_{i}=Z_{i}+\sum_{j}\alpha_{ij}Z_{j}%
\]
where the $Z_{i}$ are i.i.d. standard Gaussians$.$ Then%
\[
\mathbb{P}\left(  \frac{1}{N}\sum\nolimits_{i=1}^{N}\left\vert \zeta
_{i}\right\vert \geq K\right)  \leq\mathbb{P}\left(  \frac{1}{N}%
\sum\nolimits_{i=1}^{N}\left\vert Z_{i}\right\vert \geq K/2\right)
+\mathbb{P}\left(  \frac{1}{N}\sum\nolimits_{i=1}^{N}\left\vert Z_{i}%
\right\vert \geq\sqrt{K/2C}\right)  .
\]
By choosing $K$ appropriate, we get the desired estimate.

To prove (\ref{LD2}), we use the same representation. As%
\begin{align*}
\frac{1}{N}\sum\nolimits_{i=1}^{N}\zeta_{i}^{2}  & \leq\frac{2}{N}%
\sum\nolimits_{i=1}^{N}Z_{i}^{2}+\frac{2}{N}\sum\nolimits_{i=1}^{N}\left(
\sum\nolimits_{j}\alpha_{ij}Z_{j}\right)  ^{2}\\
& \leq\frac{2}{N}\sum\nolimits_{i=1}^{N}Z_{i}^{2}+\frac{C}{N}\left(
\sum\nolimits_{i=1}^{N}\left\vert Z_{i}\right\vert \right)  ^{2}%
\end{align*}
and%
\[
\mathbb{P}\left(  \frac{1}{N}\sum\nolimits_{i=1}^{N}Z_{i}^{2}\geq K\right)
\leq C\exp\left[  -N/C\right]
\]
for large enough $K,$ we get the desired conclusion.
\end{proof}

\begin{corollary}
\label{Le_LD} Assume $\operatorname{COND}\left(  J\right)  $ and $k\leq J.$

\begin{enumerate}
\item[a)] For any $m\leq k$ there exist $C,K>0$ such that%
\[
\mathbb{P}\left(  \frac{1}{N}\sum\nolimits_{i}\left\vert \xi_{i}^{\left(
m\right)  }\right\vert \geq K\right)  \leq C\exp\left[  -N/C\right]  .
\]

\item[b)] For any $m,l,$ there exist $C,K>0$ such that%
\[
\mathbb{P}\left(  \frac{1}{N}\sum\nolimits_{i}\left\vert \xi_{i}^{\left(
m\right)  }\xi_{i}^{\left(  l\right)  }\right\vert \geq K\right)  \leq
C\exp\left[  -N/C\right]  .
\]

\item[c)] If $Y_{i}$ are $\mathcal{F}_{m-1}$-measurable with%
\[
\mathbb{P}\left(  \sup\nolimits_{i}\left\vert Y_{i}\right\vert \geq K\right)
\leq C\exp\left[  -N/C\right]
\]
for some $K,$ then%
\[
\mathbb{P}\left(  \left\vert \left\langle \mathbf{\xi}^{\left(  m\right)
},\mathbf{Y}\right\rangle \right\vert \geq t\right)  \leq C\exp\left[
-t^{2}N/C\right]  ,\ t\leq1.
\]

\end{enumerate}
\end{corollary}

\begin{proof}
Conditioned on $\mathcal{F}_{m-1},\ \mathbf{\xi}^{\left(  m\right)  }$ is
Gaussian with covariances given by Lemma \ref{Le_G1}. On $\mathcal{F}_{m-1}%
$-measurable events $B_{N}$ with $\mathbb{P}\left(  B_{N}\right)  \geq
1-C\exp\left[  -N/C\right]  ,$ the variables appearing in this lemma on the
right hand sides are appropriately bounded. So, on $B_{N},$ the $\xi
_{i}^{\left(  m\right)  }$ are Gaussians which satisfy the conditions of the
previous lemma. So a) follows from that lemma. For b), we estimate%
\[
\frac{1}{N}\sum\nolimits_{i}\left\vert \xi_{i}^{\left(  m\right)  }\xi
_{i}^{\left(  l\right)  }\right\vert \leq\sqrt{\frac{1}{N}\sum\nolimits_{i}%
\xi_{i}^{\left(  m\right)  2}}\sqrt{\frac{1}{N}\sum\nolimits_{i}\xi
_{i}^{\left(  l\right)  2}},
\]
so that we see that it suffices to consider $l=m.$ Then we apply the lemma,
part b).

As for c), we have that the conditional distribution of $\sqrt{N}\left\langle
\mathbf{\xi}^{\left(  m\right)  },\mathbf{Y}\right\rangle $, given
$\mathcal{F}_{m-1},$ is Gaussian, with bounded variance. So the statement follows.
\end{proof}

\begin{lemma}
\label{Le_MainLLN}Let $\left\{  \eta_{i}^{\left(  N\right)  }\right\}  _{i\leq
N},$ be Gaussian vectors with $\sigma_{ij}^{\left(  N\right)  }=\mathbb{E}%
\eta_{i}^{\left(  N\right)  }\eta_{j}^{\left(  N\right)  }.$ We assume that
for some sequence $\mu_{N}>0$ with $\log\mu_{N}$ being bounded, one has%
\[
\left\vert \sigma_{ii}^{\left(  N\right)  }-\mu_{N}\right\vert \leq C/N,
\]
and there are vectors $\left\{  x_{i}^{\left(  N\right)  }\right\}  _{i\leq
N},~\left\{  y_{i}^{\left(  r,N\right)  }\right\}  _{i\leq N,\ r\leq m},\ m$
fixed, which are bounded in all indices, such that%
\[
\sup_{i\neq j,N}N^{2}\left\vert \sigma_{ij}^{\left(  N\right)  }-\frac
{x_{i}^{\left(  N\right)  }x_{j}^{\left(  N\right)  }}{N}+\sum\nolimits_{r=1}%
^{m}\frac{y_{i}^{\left(  N,r\right)  }y_{j}^{\left(  N,r\right)  }}%
{N}\right\vert <\infty.
\]
Let also $F_{N,i}$,\ $i\leq N,$ be functions $\mathbb{R\rightarrow R},$ which
are bounded and Lipshitz, uniformly in $N,i.$ Then%
\[
\frac{1}{N}\sum_{i=1}^{N}F_{N,i}\left(  \eta_{i}^{\left(  N\right)  }\right)
\simeq\frac{1}{N}\sum_{i=1}^{N}EF_{N,i}\left(  \sqrt{\mu_{N}}Z\right)  .
\]

\end{lemma}

\begin{proof}
We leave out $N$ in notations, as often as possible. Consider%
\[
\eta_{i}^{\prime}\overset{\mathrm{def}}{=}\eta_{i}+\sum_{r=1}^{m}\frac
{y_{i}^{\left(  r\right)  }}{\sqrt{N}}Z_{r}+\sqrt{K}\frac{Z_{i}^{\prime}%
}{\sqrt{N}}.
\]
The constant $K>0$ will be specified below. Then%
\begin{align*}
& \left\vert \frac{1}{N}\sum\nolimits_{i=1}^{N}F_{N,i}\left(  \eta_{i}\right)
-\frac{1}{N}\sum\nolimits_{i=1}^{N}F_{N,i}\left(  \eta_{i}^{\prime}\right)
\right\vert \\
& \leq Lc\sum_{r=1}^{m}\frac{1}{\sqrt{N}}\left\vert Z_{r}\right\vert
+L\frac{\sqrt{K}}{N^{3/2}}\sum_{i=1}^{N}\left\vert Z_{i}^{\prime}\right\vert ,
\end{align*}
where $L$ is a bound on the Lipshitz constants for the $F_{N,i},$ and $c$ is a
bound of the $\left\vert y_{i}^{\left(  r\right)  }\right\vert .$

As%
\begin{equation}
P\left(  \left\vert Z_{r}\right\vert \geq t\sqrt{N}\right)  \leq C\exp\left[
-t^{2}N/C\right]  ,\label{AppB_Est1}%
\end{equation}%
\[
P\left(  \frac{1}{N}\sum\nolimits_{i=1}^{N}\left\vert Z_{i}^{\prime
}\right\vert \geq t\sqrt{N}\right)  \leq C\exp\left[  -t^{2}N/C\right]  ,
\]
we get%
\[
\frac{1}{N}\sum_{i=1}^{N}F_{N,i}\left(  \eta_{i}\right)  \simeq\frac{1}{N}%
\sum_{i=1}^{N}F_{N,i}\left(  \eta_{i}^{\prime}\right)  .
\]%
\begin{align*}
E\left(  \eta_{i}^{\prime2}\right)   & =\mu_{N}+\delta_{i}+\sum_{r=1}^{m}%
\frac{y_{i}^{\left(  r\right)  2}}{N}+\frac{K}{N},\\
E\left(  \eta_{i}^{\prime}\eta_{j}^{\prime}\right)   & =\frac{x_{i}x_{j}}%
{N}+r_{ij},\ i\neq j.
\end{align*}
where%
\[
\delta_{i}\overset{\mathrm{def}}{=}\sigma_{ii}-\mu_{N},
\]%
\[
r_{ij}\overset{\mathrm{def}}{=}\sigma_{ij}-\frac{x_{i}x_{j}}{N}+\sum
\nolimits_{r=1}^{m}\frac{y_{i}^{\left(  r\right)  }y_{j}^{\left(  r\right)  }%
}{N}%
\]
We choose $K$ large enough such that the $N\times N$-matrix $\Gamma$ which is
$\left(  r_{ij}\right)  $ off diagonal, and%
\[
\sum\nolimits_{r=1}^{m}\frac{y_{i}^{\left(  r\right)  2}}{N}+\frac{K}{N}%
-\frac{x_{i}^{2}}{N}+\delta_{i}%
\]
on the diagonal is positive definite. This is possible as $\left\vert
r_{ij}\right\vert \leq CN^{-2}.$

Let $\left\{  U_{i}\right\}  $ be a Gaussian matrix with covariance matrix
$\Gamma.$ Then%
\[
\sqrt{\mu_{N}}Z_{i}+\frac{x_{i}}{\sqrt{N}}Z+U_{i}%
\]
has the same distribution as $\left\{  \eta_{i}^{\prime}\right\}  .$ Here we
assume that $\left\{  U_{i}\right\}  $ is independent of the $Z$'s. So, we
assume that the $\eta_{i}^{\prime}$ are presented in this way.%
\[
\left\vert \frac{1}{N}\sum\nolimits_{i=1}^{N}F_{N,i}\left(  \eta_{i}^{\prime
}\right)  -\frac{1}{N}\sum\nolimits_{i=1}^{N}F_{N,i}\left(  \sqrt{\mu_{N}%
}Z_{i}\right)  \right\vert \leq CL\frac{\left\vert Z\right\vert }{\sqrt{N}%
}+L\frac{1}{N}\sum_{i=1}^{N}\left\vert U_{i}\right\vert .
\]
We can apply Lemma \ref{Le_Gauss_App1} to the vector $\left(  \sqrt{N}%
U_{i}\right)  _{1\leq i\leq N}$, and (\ref{AppB_Est1}) to the first summand on
the right-hand side, obtaining%
\[
\frac{1}{N}\sum\nolimits_{i=1}^{N}F_{N,i}\left(  \eta_{i}^{\prime}\right)
\simeq\frac{1}{N}\sum\nolimits_{i=1}^{N}F_{N,i}\left(  \sqrt{\mu_{N}}%
Z_{i}\right)
\]%
\[
\frac{1}{N}\sum\nolimits_{i=1}^{N}F_{N,i}\left(  \sqrt{\mu_{N}}Z_{i}\right)
\simeq\frac{1}{N}\sum\nolimits_{i=1}^{N}EF_{N,i}\left(  \sqrt{\mu_{N}%
}Z\right)  ,
\]
follows by standard Gaussian isoperimetry (see e.g. \cite{LeTala}).
\end{proof}

\end{document}